\documentclass[runningheads]{llncs}
\usepackage[T1]{fontenc}
\usepackage{graphicx}
\usepackage{amssymb}
\usepackage{mathrsfs}
\usepackage{stmaryrd}
\usepackage{xcolor}
\usepackage{amsmath,stackengine}
\usepackage{array}
\usepackage{hyperref}
\newtheorem{notation}[remark]{Notation}

\input{tikzit.tikzdefs}
\usepackage{tikzit}

\tikzstyle{empty}=[shape=circle, tikzit fill={rgb,255: red,191; green,191; blue,191}]
\tikzstyle{dot}=[fill=black, draw=black, shape=circle]

\tikzstyle{to}=[draw=black, ->]
\tikzstyle{equal-arrow}=[{-}, double equal sign distance]
\tikzstyle{hook}=[right hook->, draw=black, tikzit draw=magenta]
\tikzstyle{dashed-arrow}=[dashed, ->]
\tikzstyle{mono}=[draw=black, to reversed->]

\begin{document}
\title{On the Centre of Strong Graded Monads}
\author{Flavien Breuvart\inst{1} \and
Quan Long\inst{2}\and
Vladimir Zamdzhiev\inst{3}}
\authorrunning{F. Breuvart et al.}
%
\institute{LIPN, CNRS, Université Paris Nord
 \and
  École Normale Supérieure Paris-Saclay
 \and
  Université Paris-Saclay, CNRS, ENS Paris-Saclay, Inria, Laboratoire Méthodes Formelles, 91190, Gif-sur-Yvette, France
}
\maketitle              
\begin{abstract}
  We introduce the notion of ``centre'' for pomonoid-graded strong
  monads which generalizes some previous work that describes the centre of
  (not graded) strong monads. We show that, whenever the centre exists, this
  determines a pomonoid-graded commutative submonad of the original one. We
  also discuss how this relates to duoidally-graded strong monads.
\end{abstract}
%
%
%
\section{Introduction and Related Work}
\label{sec:intro}

The notions of \emph{centre/centrality} and similarly
\emph{commutant/centraliser} can be formulated for many different kinds of
algebraic structures, e.g. monoids, groups, semirings. It also makes sense in
certain kinds of categorical settings. For instance, premonoidal categories
admit a centre \cite{premonoidal}, which is essential for the development of
the theory. Another example is given by enriched algebraic theories and
associated monads which were shown to admit centralisers in \cite{commutants}.
In other related work \cite{commutativity}, the authors study commutativity in
a duoidal setting which is related to the above mentioned notions of centre and
centralisers. More recently, in \cite{central-submonads}, the authors
established some additional results for the centre of a strong monad on a
symmetric monoidal category and also studied the more general concept of
central submonad. Inspired by these developments, in this paper we consider yet
another notion of centre, this time for strong monads that are graded by a
(partially ordered) monoid, which can be seen as an immediate generalization of
the definition of centre proposed in \cite{central-submonads}.

Constructing the centre of a monad is a way to recover the commutativity of it
by removing the effects that violate it. In practice, commutativity of a monad
is an important property which means that the effect can occur
deterministically inside operations without forcing the evaluation order (most
compilers are allowed to evaluate operands in the order of their choice for
optimization purposes). A Graded monad, however, can be seen as a form of
statistical analysis that gives insight on the effects that can occur. We think
that studying the interaction between these structures is important.

Before we present the technical results, we provide some computational
intuition for the centre of a strong monad that may be useful to some readers.
Strong monads on a symmetric monoidal category can be used to represent
sequencing of computational effects.  Informally, suppose we are given an
effectful computation $f \colon X_1 \to X_2$ acting on some variable $x_1 \colon X_1$ and another
effectful computation $g \colon Y_1 \to Y_2$ acting on a different variable $y_1 \colon Y_1$. Since
these two computations are effectful, the two ways of
sequencing the computations
\[ \text{do } x_2 \leftarrow f(x_1); y_2 \leftarrow g(y_1); h(x_2, y_2) \]
and
\[ \text{do } y_2 \leftarrow g(y_1); x_2 \leftarrow f(x_1); h(x_2, y_2) \]
do \emph{not} necessarily have the same computational result, even though the
two computations are acting on different variables that have different types.
When we interpret the above computational situation in the Kleisli category
$\CC_T$ of a strong monad $T \colon \CC \to \CC$, this potential difference is
reflected by the fact that $\CC_T$ has a \emph{premonoidal} structure
\cite{premonoidal}, rather than a monoidal one, and the premonoidal product is
not bifunctorial (in general), unlike the monoidal one. When the effect under
consideration is commutative and so the monad $T$ is also \emph{commutative}
(not just strong), then the two ways of sequencing above have the same
computational result and this is reflected in the Kleisli category $\CC_T$
which has a monoidal structure (not just premonoidal) in this case. So, we may
naturally arrive at the notion of \emph{centre} of a strong monad by
identifying all the central elements/effects, i.e. those that commute with
all other elements/effects. This is the approach taken in
\cite{central-submonads} and in this paper we consider a more general setting.

More specifically, we consider a wider range of effects, namely those that can
be described by pomonoid-graded strong monads, and then we propose a definition
for the centre of such monads. The construction is analogous, but more general,
compared to the one in \cite{central-submonads}.
In the last section, we investigate possible research perspectives which exploit
our understanding of the preliminary study performed here. In particular, we
acknowledge the fact that the centre of graded monads is seldom usable in practice as
it is even more constrained than that of a monad. Finally, we open the discussion to
relaxations in which we may eventually obtain more refined notions of centres.

\section{Graded Monads}
\label{sec:graded}

There are different ways of introducing graded monads in the literature~\cite{FujiiKM16,canonical,OWE20}.
They all share the same structure: using a monoid to index functors which, together, behave like a monad whose multiplication is graded by that of the monoid. 
It is then possible to extend the monoid with additional structures such as that of a semiring to represent additive monads, for example. 
We present the most common of these extensions, that of gradations by pomonoids (partially ordered monoids), 
whose order represents a degree of knowledge and which is essential for most applications that use statistical analysis \cite{IMO20}.

\subsection{Graded Monads}
\label{sub:def-graded}

Before we present the full definition of a pomonoid-graded monad, we start with
a monoid-graded monad so that readers can hopefully acquire a better intuition
for the more general notion that follows afterwards.

\begin{definition}[Monoid Graded Monads]\label{def:graded-monad}
	Let $ \GG = (G, i, \ast)$ be a monoid. 
	A $\GG$-graded monad on a category $\CC$ is given by the following data:
	\begin{itemize}
		\item for any $\textcolor{red}{a} \in G $, an endofunctor $\overset{\textcolor{red}{a} }{\TT} : \CC \rightarrow \CC$;
		\item a natural transformation $\eta : \Id \rightarrow \overset{ i }{\TT}$;
		\item for any $\textcolor{red}{a}, \textcolor{teal}{b} \in G $, a natural transformation	
		$\mu^{\textcolor{red}{a}, \textcolor{teal}{b}}: \overset{\textcolor{red}{a} }{\TT} 
		\cdot \overset{\textcolor{teal}{b} }{\TT} \rightarrow \overset{\textcolor{red}{a} \ast \textcolor{teal}{b} }{\TT}$,
      where ($- \cdot -$)  indicates functor composition (we sometimes omit the gradations on $\mu$ for brevity);
	\end{itemize}
	such that the following diagrams commute:
	\[\stikz[0.83]{../figures/monad-def-graded.tikz}\]
\end{definition}

\begin{remark}
  Every monad $T \colon \CC \to \CC$ can be seen as a $\GG$-graded monad by
  setting $\GG$ to be a singleton (i.e. trivial) monoid.
\end{remark}

Next, we introduce pomonoids and the appropriate notion of morphism between
them. This would then allow us to introduce a more flexible notion of gradation
for our monads.

\begin{definition}[Pomonoid \cite{perf_pomo}]
  A \emph{partially ordered monoid (pomonoid)} is a tuple $((G, \leq), i, \ast ),$ where $(G, i, \ast )$ is a monoid and where $\leq$ is a partial order on $G$, such that
  the monoid operation is monotone with respect to the order in the following sense:
	for all $ w,x,y,z \in G $, if $w \leq  x$  and $y \leq  z$, then $w \ast y \leq x \ast z$.
\end{definition}

\begin{definition}[Morphism between Pomonids]\label{def:morph-pomonoid}
	Let $\GG = ((G, \leq), i, \ast)$ and $\mathcal H = ((H, \sqsubseteq), e, \circledast)$ be two pomonoids. 
	A morphism between pomonoids is a function $\phi : G \to H,$ such that:
	\begin{itemize}
		\item  $e \sqsubseteq \phi (i)$
		\item  $\phi(x) \circledast \phi(y) \sqsubseteq \phi (x \ast y)$ for all $x,y \in G.$
	\end{itemize}
\end{definition}

\begin{remark}
  A pomonoid $((G, \leq), i, \ast )$ can also be seen as a monoidal category $\CC$ that is 
  both skeletal and thin in the following way:
  \begin{itemize}
    \item objects of $\CC$ are given by the elements $a \in G$;
    \item $\CC$ has a unique morphism $a \to b$  
      iff $a \leq b$ in $G$; 
    \item the tensor product is given by $a \otimes b \eqdef a \ast b$;
    \item the tensor unit is given by $i$.
  \end{itemize}
  This correspondence may also be extended to cover morphisms between
  pomonoids: if $\phi \colon \GG \to \HH$ is a morphism of pomonoids, then
  $\phi$ may be identified with a lax monoidal functor between the (thin and
  skeletal) monoidal categories that represent $\GG$ and $\HH$.
\end{remark}

\begin{definition}[Pomonoid Graded Monad]\label{def:graded-monad-pomonoid}
  A \emph{pomonoid-graded monad} on a category $\CC$ is given by the following data:
  \begin{itemize}
    \item a pomonoid $\GG = ((G, \leq), i, \ast),$ whose elements $\textcolor{red}{a} \in G$ we call \emph{gradations},
    \item for any $\textcolor{red}{a} \in G $, an endofunctor $\overset{\textcolor{red}{a} }{\TT} : \CC \rightarrow \CC$;
    \item a natural transformation $\eta : \Id \rightarrow \overset{ i }{\TT}$;
    \item for any $\textcolor{red}{a}, \textcolor{teal}{b} \in G $, a natural transformation	
      $\mu^{\textcolor{red}{a}, \textcolor{teal}{b}}: \overset{\textcolor{red}{a} }{\TT} 
      \cdot \overset{\textcolor{teal}{b} }{\TT} \rightarrow \overset{\textcolor{red}{a} \ast \textcolor{teal}{b} }{\TT}$
    \item such that the following diagrams commute:
    \[\stikz[0.83]{../figures/monad-def-graded.tikz}\]
    \item for any $\textcolor{red}{a} \leq  \textcolor{red}{a'}$ in $\GG$, 
      a natural transformation $\overset {\textcolor{red}{a} \leq \textcolor{red}{a'}}{\TT} : \overset{\textcolor{red}{a} }{\TT}\to \overset{\textcolor{red}{a'} }{\TT}$, such that:
    \[	
      \overset{\textcolor{red}{a} \leq \textcolor{red}{a}} {\TT} = \iota \ (\text{identity natural transformation}) \; \; \; , 
      \overset{\textcolor{red}{a'} \leq \textcolor{red}{a''}} {\TT} \circ \overset{\textcolor{red}{a} \leq \textcolor{red}{a'}} {\TT} 
      =
      \overset{\textcolor{red}{a} \leq \textcolor{red}{a''}} {\TT} \text{ and }
      \]
      \[\stikz[0.83]{../figures/monad-def-graded-pomonoid.tikz}\]
	\end{itemize}
\end{definition}

\begin{remark}
  $\TT$ can be seen as a bifunctor $\GG\times \CC \rightarrow \CC$, with the last conditions being the functoriality and the naturality of $\mu$ with respect to $\GG$.
  In fact, we can go even further: graded monads over $\GG$ on a category $\CC$ are exactly lax-monoidal functors $(\TT,\mu,\eta)$ from $(\GG,*,i)$, seen as a thin monoidal category, to $([\CC,\CC],\cdot,\id)$. 
\end{remark}

\begin{example}
  Let $T$ be a strong monad and $Z$ its centre, then we can always construct a monad graded by $((\mathtt{Bool},\mathtt{t\!\!t}\le\mathtt{f\!\!f}),\mathtt{t\!\!t},\wedge)$ defined by $\overset{\mathtt{t\!\!t}}T=Z$ and $\overset{\mathtt{f\!\!f}}T=T$ that keeps track whether we are in the centre or not. Notice that the morphism $\overset{\mathtt{t\!\!t}\le\mathtt{f\!\!f}}T:Z\rightarrow T$ is the submonad monomorphism (inclusion).
  Such gradations can be used to infer whether an operation $\odot:A\times B\rightarrow C$ can be lifted into an effectful environment $\hat \odot:TA\times TB\rightarrow TC$ and still be evaluated freely by the compiler. Indeed, the graded version is of the form $\hat \odot:\overset{\mathtt{a}}TA\times \overset{\mathtt{b}}TB\rightarrow \overset{\mathtt{a\wedge b}}TC$, the evaluation order have to be forced if $a=b=\mathtt{f\!\!f}$ but is commutative otherwise.
\end{example}

The notion of centre (and centrality) can be formulated for \emph{strong}
monads, as was previously argued in \cite{central-submonads}. The corresponding
definition for pomonoid-graded monads is presented next.

\begin{definition}[Strong Pomonoid Graded Monad]\label{def:strong-monad-graded}
	Let $ \GG = ((G, \leq), i, \ast)$ be a pomonoid.
  A \emph{strong} $\GG$-graded monad over a monoidal category $(\CC,\otimes,I,\alpha,\lambda, \rho)$ 
	is a $\GG$-graded monad $(\TT ,\eta,\mu)$ equipped with a family of natural transformations
	$\overset{\textcolor{red}{a}}{\tau}_{X,Y}:X\otimes \overset{\textcolor{red}{a}}{\TT} Y\to \overset{\textcolor{red}{a}}{\TT} (X\otimes Y) , $
	indexed by elements $\textcolor{red}{a} \in \GG$, called \emph{graded strength},
	such that for any element $\textcolor{teal}{b} \in G$ and objects $X, Y$ in $\CC$, the following diagrams commute:
	\[\stikz[0.83]{../figures/strength-def-graded.tikz}\]
	where (for simplicity) we omitted the superscripts of $\tau$.
\end{definition}

If, moreover, we are given a \emph{symmetric} monoidal category $(\CC, \otimes, I, \gamma)$, then
we can define the $\GG$-\emph{graded costrength}
$\overset{\textcolor{red}{a}}{\tau'}_{X,Y} \colon \overset{\textcolor{red}{a}}{\TT} X \otimes Y \to \overset{\textcolor{red}{a}}{\TT}(X \otimes Y)$ 
by $\overset{\textcolor{red}{a}}{\tau'}_{X,Y} \eqdef \overset{\textcolor{red}{a}}{\TT}(\gamma_{Y,X})\circ \overset{\textcolor{red}{a}}{\tau}_{Y,X}\circ\gamma_{\overset{\textcolor{red}{a}}{\TT} X,Y}$,
where $\gamma$ represents the symmetry. Note that this is completely analogous
to how the costrength is defined without for monads that are not graded. We also often omit the
gradations on $\tau'$ for convenience and note that it satisfies similar
coherence conditions to that of the strength $\tau.$

We can now introduce commutative graded monads whose definition is also
analogous to the definition of a commutative monad that is not graded.

\begin{definition}[Commutative Graded Monad]
	\label{def:commutative-monad-graded}
  Let $\GG = ((G, \leq), i, \ast)$ be a pomonoid (Not necessarily commutative).
	Let $(\TT, \eta, \mu, \tau)$ be a strong $\GG$-graded monad on 
	a \emph{symmetric} monoidal category $(\CC, \otimes, I, \gamma)$.
	Then, $\TT$ is said to be \emph{commutative} if for any $\textcolor{red}{a},\textcolor{teal}{b}\in G$,
	and any objects $X, Y \in \CC $, the following diagram commutes:
	  \[\label{eq:commutative-monad-graded}
	  \stikz[0.83]{../figures/commutative-monad-def-graded.tikz}\]
\end{definition}

\begin{remark}
  Note that in the above definition, the pomonoid is not necessarily assumed to be commutative,
  so that $\textcolor{red}{a} \ast \textcolor{teal}{b} = \textcolor{teal}{b}
  \ast \textcolor{red}{a} $ which justifies the equality in the bottom right
  corner.
\end{remark}

In the previous attempts, we tried with pomonoid who are already commutative, which we found to be too limited. 
Imagine such a scenario, that an identity monad is graded by any pomonoid, the diagram will always commute. 
But it will not be captured by the previous definition, if one believes that such case should be included, then 
the gradation should not by definition be already commutative.

Also, \cite{commu-graded-22} provided a definition of commutativity in graded monads, but from our understanding, 
this commutativity is only describing the identical case where $\textcolor{red}{a} \ast i = i \ast \textcolor{red}{a}$, 
it does not capture the general case of $\textcolor{red}{a}\ast \textcolor{teal}{b} = \textcolor{teal}{b}\ast \textcolor{red}{a}$,
hence within our knowledge, we gave the first complete definition of the commutativity of graded monads.

\subsection{Morphisms between Strong Graded Monads}
\label{sub:morphism-graded}
One of the issues that we have to address is how to formulate an appropriate
definition of the morphisms between strong graded monads. This is important for
our development as it underpins subsequent constructions that are relevant for
the construction of the centre.


In order to visualize the definition of morphism between a strong $\GG$-graded
monad $T$ and a $\HH$-graded monad $P$ over a category $\CC$, remember that
they can be seen as lax monoidal functors $\TT:\GG\rightarrow [\CC,\CC]_s$ and
$\PP:\mathcal H\rightarrow [\CC,\CC]_s$ targeting the same monoidal category,
that of endofunctors $[\CC,\CC]_s$ over $\CC$. Therefore, the morphisms between
them can be identified with the morphisms of the lax-slice category over
$[\CC,\CC]_s$ in $\mathtt{MonCat}$.

\begin{center}
\begin{tikzpicture}
  \node (G) at (0,4) {$(\GG,*,i)$};
  \node (G') at (0,0) {$(\mathcal H,\circledast,e)$};
  \node (C) at (10,2) {$([\CC,\CC]_s,\cdot,I)$};
  \node (A1) at (2.7,2.5) {};
  \node (A2) at (2.4,1) {};
  \draw[->] (G) to node [auto] {{\scriptsize $(\TT,\mu^\TT,\nu^\TT)$}} (C);
  \draw[->] (G') to node [below right] {{\scriptsize $(\PP,\mu^\PP,\nu^\PP)$}} (C);
  \draw[->] (G) to node [left] {{\scriptsize $(\phi,\mu^\phi,\nu^\phi)$}} (G');
  \draw[->,double] (A1) to node [right] {$\iota$} (A2);
\end{tikzpicture}
\end{center}

\begin{definition}[Morphism between Strong Graded Monads]
	\label{def:morph-monad-graded}
	Let $\GG = ((G,\leq),i,\ast)$ and  $\mathcal H = ((H,\sqsubseteq), e, \circledast)$ be two pomonoids.
	Let $(\TT, \eta^\TT, \mu^\TT, \tau^\TT)$ be a $\GG$-graded strong monad and let
	$(\PP,\eta^\PP, \mu^\PP, \tau^\PP)$ be an $\mathcal H$-graded strong monad  
	over a symmetric monoidal category $\CC = (C, \otimes, I, \gamma)$. 
  A \emph{morphism of strong graded monads} is given by the following data: 
  \begin{itemize}
    \item a pomonoid-morphism $\phi \colon \GG \to \mathcal H$ between the gradations;
    \item a family of natural transformations
    $\overset{\textcolor{red}{a}}{\iota} : \overset{\textcolor{red}{a}}{\TT} \Rightarrow \overset{ \phi \textcolor{red}{a} }{\PP}$ 
    indexed by elements of $\GG$\footnote[2]{We often omit the superscript of $\iota$ for convenience.},
  \item such that for all $\textcolor{red}{a}, \textcolor{teal}{b}$ in $\GG$ and $X,Y$ in $\CC$,
    the following diagrams commute: 
    \[\stikz[0.83]{../figures/map-of-monads-graded-def-1-pomonoid.tikz}\]
  \end{itemize}
\end{definition}

The subtlety here is to explain the relations between 
$\overset{e}{\PP}$, $\overset{\phi i}{\PP},$
$\overset{\phi \textcolor{red}{a} \circledast \phi \textcolor{teal}{b}}{\PP}$ and 
$\overset{\phi ( \textcolor{red}{a} \ast \textcolor{teal}{b})}{\PP}$.
This comes from the fact that the inequalities between the gradations induce natural transformations between the indexed endofunctors.

\begin{remark}
  If we equip the pomonoids $\GG$ and $\HH$ with the discrete order, then the inequalities under consideration become the usual
  equalities for a monoid homomorphism, i.e. $\phi(i) = e$ and $\phi (\textcolor{red}{a} \ast \textcolor{teal}{b}) = \phi \textcolor{red}{a} \circledast \phi \textcolor{teal}{b}$.
  Then we can obviously identify $\overset{\phi i} \PP X = \overset{e} \PP X$ and
  $\overset{\phi (\textcolor{red}{a} \ast \textcolor{teal}{b})} \PP X 
  = \overset{\phi \textcolor{red}{a} \circledast \phi \textcolor{teal}{b}} \PP X$, so that we recover a notion of morphism of strong monoid-graded monads as a special case.
\end{remark}

\section{The Centre of a Strong Graded Monad}
\label{sec:graded-centre}
We can now outline our construction for the centre of a strong graded monad.
Our construction is similar to the construction of the centre of a strong (not
graded) monad in \cite{central-submonads}, but some of our proofs are
established in a slightly different way.  For example, some of the required
proofs in \cite{central-submonads} use results from the theory of premonoidal
categories \cite{premonoidal}. However, in our setting, this is not so
straightforward, because in order to adapt the proofs, we would have to
identify a suitable ``premonoidal'' theory for strong graded monads. Instead of
doing this, we opt for a more direct approach that boils down to fairly large
diagrammatic chasing in a few cases (see Appendix \ref{sec:app}).
In particular, this gives us (as a special case) another proof for some of the
established results in \cite{central-submonads}. 

A basic intuition that we have for the centre of any algebraic structure is
that it enjoys commutativity properties. The definition of a commutative graded
monad (Definition \ref{def:commutative-monad-graded}) suggests that we should
also take the centre of the gradations into account as well. This brings us to
our next definition.

\begin{definition}[Centre of a Pomonoid]\label{def:centre-monoid}
  Given a pomonoid $\GG = ((G,\leq) , i, \ast)$, the \emph{centre} of the pomonoid, written $Z(G)$, is given by the set
	\[ 
    Z(G) \eqdef \{ \textcolor{red}{z}\in G  |\ \forall \textcolor{teal}{b} \in G, \
			\textcolor{red}{z} \ast \textcolor{teal}{b} = 
	\textcolor{teal}{b} \ast \textcolor{red}{z} \} .
	\]
\end{definition}

It is easy to see that $\ZZ(\GG) \eqdef ((Z(G),\leq), i, \ast )$ is a subpomonoid of
$\GG$ in the sense that the subset inclusion $\phi \colon Z(G) \subseteq G$ is a
pomonoid-morphism.

\begin{remark}
  The centre of a pomonoid, as given above, does not take the order into
  consideration.  It will be interesting to identify a more general and flexible
  notion of centre that depends on the order and use that in future work.
\end{remark}

Next we explain how to construct the centre of strong graded monads on $\Set$,
which we hope would help readers in understanding the more general construction
that follows afterwards.

\begin{definition}[Graded Centre in $\Set$]\label{def:centre-set-graded}
	Let $\ZZ (\GG) = ((Z(G),\leq),i,\ast)$ be a pomonoid which is the centre of the pomonoid $\GG = ((G,\leq),i,\ast)$.
	Assume further that we are given a strong $\GG$-graded monad $(\TT, \eta, \mu, \tau)$ on $\Set$.

  Given an arbitrary $\textcolor{red}{z}\in Z(G)$ and set $X$, 
	we say that the $\ZZ (\GG)$-graded \emph{centre} of $\TT$ at $(\textcolor{red}{z},X)$, 
	written $\overset{\textcolor{red}{z}}{\ZZ} X$, is the set 
	\[
        \overset{\textcolor{red}{z}}{\ZZ} X \eqdef
        \left\{ t \in \overset{\textcolor{red}{z}}{\TT} X \ \middle |\
        \begin{matrix}
	  \forall \textcolor{teal}{b} \in G,
	  \forall Y \in \Ob(\Set), \forall s \in \overset{\textcolor{teal}{b}}{\TT} Y, \\
	   \mu(\overset{\textcolor{red}{z}}{\TT}\tau'(\tau(t,s))) = 
	   \mu(\overset{\textcolor{teal}{b}}{\TT}\tau(\tau'(t,s)))
        \end{matrix}
           \right\} .
	\]
	We write $\overset{\textcolor{red}{z}}{\iota}_X : 
	\overset{\textcolor{red}{z}}{\ZZ} X \subseteq \overset{\textcolor{red}{z}}{\TT} X$ 
	for the indicated subset inclusion.
\end{definition}

The main idea for this definition is to consider all of the monadic elements
that satisfy the equation in Definition \ref{def:commutative-monad-graded}, for
suitably fixed $X$ and $\textcolor{red}{z}\in Z(G)$.

\begin{notation}
For the rest of this section, we assume that $\CC$ is a symmetric monoidal category; $V,W,X,Y$ are objects in $\CC$;
we are given a pomonoid $\GG = ((G,\leq),i,\ast)$ with centre $\ZZ (\GG) = ((Z(G),\leq),i,\ast)$;
we also have a $\GG$-graded strong monad $\TT$; 
we write $\textcolor{red}{z} \in Z(G)$ for central gradations and more generally we write 
$\textcolor{teal}{b}, \textcolor{brown}{c} \in G$ for gradations of $\GG$.
\end{notation}

To extend the definition of the centre from $\Set$ to other categories, we
introduce graded central cones.

\begin{definition}[Graded Central Cone]\label{def:central-cone-graded}
	Let $X$ be an object of $\CC$ and $\textcolor{red}{z}\in Z(G)$.
	A \emph{graded central cone} of a $\GG$-graded strong monad $\TT$ at $(\textcolor{red}{z},X)$,
	is given by a pair $(Z, \iota)$ of an object $Z$ and a morphism
	$\iota \colon Z  \to \overset{\textcolor{red}{z}}{\TT} X$,
	such that for any object $Y$ in $\CC$ and any $\textcolor{teal}{b}\in G$, the following diagram commutes:
	\[\begin{tikzpicture}
	\begin{pgfonlayer}{nodelayer}
		\node [style=empty] (0) at (-6.75, 4) {$Z \otimes\overset{\textcolor{teal}{b}}{\TT} Y$};
		\node [style=empty] (1) at (3, 4) {$\overset{ \textcolor{red}{z}}{\TT} X\otimes\overset{\textcolor{teal}{b}}{\TT} Y$};
		\node [style=empty] (2) at (12, 4) {$\overset{ \textcolor{red}{z}}{\TT} (X\otimes\overset{\textcolor{teal}{b}}{\TT} Y)$};
		\node [style=empty] (3) at (12, -0.5) {$\overset{ \textcolor{red}{z}}{\TT}  \overset{\textcolor{teal}{b}}{\TT} (X\otimes Y)$};
		\node [style=empty] (4) at (8.75, -7.5) {$\overset{\textcolor{teal}{b} \ast  \textcolor{red}{z}}{\TT}  (X\otimes Y)$};
		\node [style=empty] (5) at (-6.75, -1.5) {$\overset{ \textcolor{red}{z}}{\TT} X\otimes\overset{\textcolor{teal}{b}}{\TT} Y$};
		\node [style=empty] (6) at (-6.75, -7.5) {$\overset{\textcolor{teal}{b}}{\TT}(\overset{ \textcolor{red}{z}}{\TT}X\otimes Y)$};
		\node [style=empty] (7) at (1.5, -7.5) {$\overset{\textcolor{teal}{b}}{\TT} \overset{ \textcolor{red}{z}}{\TT} (X\otimes Y)$};
		\node [style=empty] (8) at (-6.75, 3.5) {};
		\node [style=empty] (9) at (-6.75, -1) {};
		\node [style=empty] (10) at (-6.75, -2) {};
		\node [style=empty] (11) at (-6.75, -7) {};
		\node [style=empty] (12) at (12, 3.5) {};
		\node [style=empty] (13) at (12, 0) {};
		\node [style=empty] (14) at (12, -1) {};
		\node [style=empty] (15) at (12, -4.5) {};
		\node [style=empty] (16) at (-2, 4.75) {$\iota\otimes\overset{\textcolor{teal}{b}}{\TT} Y$};
		\node [style=empty] (17) at (7, 4.5) {$\tau'$};
		\node [style=empty] (18) at (13, 1.75) {$\overset{ \textcolor{red}{z}}{\TT}\tau$};
		\node [style=empty] (19) at (13.25, -2.75) {$\mu_{X\otimes Y}^{ \textcolor{red}{z} , \textcolor{teal}{b}}$};
		\node [style=empty] (20) at (-8, 1.25) {$\iota\otimes\overset{\textcolor{teal}{b}}{\TT} Y$};
		\node [style=empty] (21) at (-7.5, -4.5) {$\tau$};
		\node [style=empty] (22) at (-2.75, -8.25) {$\overset{\textcolor{teal}{b}}{\TT}\tau'$};
		\node [style=empty] (23) at (5.25, -8.25) {$\mu_{X\otimes Y}^{\textcolor{teal}{b} ,   \textcolor{red}{z}}$};
		\node [style=empty] (28) at (12, -5) {$\overset{ \textcolor{red}{z} \ast \textcolor{teal}{b}}{\TT}  (X\otimes Y)$};
		\node [style=none] (29) at (10, -6.75) {};
		\node [style=none] (30) at (11.25, -5.75) {};
	\end{pgfonlayer}
	\begin{pgfonlayer}{edgelayer}
		\draw [style=to] (0) to (1);
		\draw [style=to] (1) to (2);
		\draw [style=to] (6) to (7);
		\draw [style=to] (7) to (4);
		\draw [style=to] (8) to (9);
		\draw [style=to] (10) to (11);
		\draw [style=to] (14) to (15);
		\draw [style=to] (12) to (13);
		\draw [style=equal-arrow] (29.center) to (30.center);
	\end{pgfonlayer}
\end{tikzpicture}\]

	If $(Z', \iota')$ and $(Z, \iota)$ are two graded central cones of $\TT$ at $(\textcolor{red}{z},X)$,
	a \emph{morphism of graded central cones} 
	$\varphi : (Z', \iota') \to 
	(Z, \iota)$ is a morphism $\varphi : Z' \to Z,$ such that $\iota \circ \varphi = \iota'.$ 
	Graded central cones of $\TT$ at $(\textcolor{red}{z},X)$ form a category
  and a \emph{terminal graded central cone} of $\TT$ at $(\textcolor{red}{z},X)$
  is a terminal object in that category.
\end{definition}

\begin{proposition}
\label{prop:unique-graded}
  If a terminal graded central cone for a $\GG$-graded strong monad $\TT$ at
  $(\textcolor{red}{z},X)$ exists, then it is unique up to a unique isomorphism
  of graded central cones. Also, if $(Z, \iota)$ is a terminal graded central
  cone, then $\iota$ is a monomorphism.
\end{proposition}
\begin{proof}
  Straightforward, essentially the same as in \cite{central-submonads}.
\end{proof}

In particular, Definition \ref{def:centre-set-graded} gives a terminal graded central cone
for the special case of graded monads over $\Set$.

\begin{definition}[Centralisable Graded Monad]
\label{def:centralisable-graded}
	We say that the $\GG$-graded monad $\TT$ over $\CC$ is \emph{centralisable} if, for any object $X$ in $\CC$, 
	for any element $\textcolor{red}{z}$ in $\ZZ(\GG)$, 
	a terminal graded central cone of $\TT$ at $(\textcolor{red}{z},X)$ exists. 
	In this situation, we write
	$(\overset{\textcolor{red}{z}}{\ZZ} X, \overset{\textcolor{red}{z}}{\iota} _X)$ 
	for the terminal graded central cone of $\TT$ at $(\textcolor{red}{z},X)$.
\end{definition}

For a centralisable $\GG$-graded monad $\TT$, the next theorem shows that its
terminal graded central cones induce a commutative $\ZZ(\GG)$-graded submonad
$\ZZ$ of $\TT$, which we call the \emph{centre} of $\TT$.

\begin{theorem}[Centre]\label{th:submonad-from-cone-graded}
	If the $\GG$-graded monad $\TT$ is centralisable, then the assignment $\ZZ(-)$ extends to a
	commutative $Z(\GG)$-graded monad $(\ZZ, \eta^\ZZ, \mu^\ZZ, \tau^\ZZ)$ on $\CC$, called the \emph{centre} of $\TT$. 
	Moreover, $\ZZ$ is a commutative $\ZZ(\GG)$-graded submonad of $\TT$ and the family of morphisms 
	$\overset{\textcolor{red}{z}}{\iota}_X : \overset{\textcolor{red}{z}}{\ZZ} X
	\to \overset{\textcolor{red}{z}}{\TT} X$, 
	determine a monomorphism of strong graded monads 
	$\ZZ \to \TT$ in the sense of Definition \ref{def:morph-monad-graded}.
\end{theorem}

\section{Examples of Centres of Strong Graded Monads}
\label{sec:examples}

In this section we provide more examples of strong graded monads that admit
centres. We begin with a concrete example in $\Set$ and we discuss the role of
the gradations.

\begin{example}\label{ex:writer-monad-graded}
  The multi-error graded writer monad is a writer monad which stops the
  computation whenever it encounters an error.
	For the gradation, consider a monoid $\GG = (\{t, e,w_1...w_n\},t,\ast)$ of three kinds of possible outcomes: 
	$t$ indicates the result is not an error nor warning, 
	$w_{i}$ are arbitrary elements in the set of warnings, and
	$e$ is an error. The neutral element is $t.$

	The monoid composition is defined as follows:
	\begin{center}
	\setlength\extrarowheight{3pt}
	\noindent\begin{tabular}{c | c c c c}
		$\ast$ & $t$ & $e$ & $w_{{a}}$ & $w_{{b}}$ \\
		\cline{1-5}
		$t$ & $t$ & $e$ & $w_{{a}}$ & $w_{{b}}$\\
		$e$ & $e$ & $e$ & $e$ & $e$ \\
		$w_{{a}}$ & $w_{{a}}$ & $e$ & $w_{{a}}$ & $w_{{b}}$\\
		$w_{{b}}$ & $w_{{b}}$ & $e$ & $w_{{a}}$ & $w_{{b}}$\\
	\end{tabular}
	\end{center}
	The rules $w_{{a}}\ast w_{{b}} = w_{{b}}$ and $w_{{b}} \ast w_{{a}} = w_{{a}}$
	indicate that we only keep track of the first warning encountered. Thus the centre of $\GG$ is $\ZZ(\GG) = (\{ t, e\}, t, \ast)$.

	Now we can define the multi-error writer monad as the writer monad $\TT$ graded by $\GG$, 
  corresponding to the different annotations of outcomes:
    \[ \overset{e}{\TT} X \defeq 1,\quad \overset{t}{\TT} X \eqdef X,\quad \overset{w_{{a}}}{\TT} X \eqdef X \times \{{a}\}\quad \overset{w_{{b}}}{\TT} X \eqdef X \times \{{b}\}. \]
  The unit is given by $\eta(x)=x\in\overset{t}{\TT}$ and the multiplication by
    \[ \mu_{w_i,w_j}((x,i),j) \eqdef i,\quad \mu_{t,w_i}(x,i) \eqdef \mu_{w_i,t}(x,i) \eqdef (x,i),\quad \mu_{t,t}(x) \eqdef x. \]
	Its centre is a commutative $(\{ t, e\}, t, \ast)$-graded submonad.
	It consists of two endofunctors that are given by $\overset{e}{\TT} X = 1$ and $\overset{t}{\TT} X = X$.
\end{example}

The category $\Set$ is not the only category for which we can systematically construct centers of strong graded monads.
Just like in \cite{central-submonads}, there are many such categories. Next, we provide such an example.

\begin{example}
	Let $T$ be a strong $\GG$-graded monad on $\TOP$, the category of topological spaces and continuous maps between them.
  The strength is considered with respect to the Cartesian structure of $\TOP$. Then $T$ is centralisable with
  terminal central cones given by
	\[
        \overset{\textcolor{red}{z}}{\ZZ} X \eqdef
        \left\{ t \in \overset{\textcolor{red}{z}}{\TT} X \ \middle |\
        \begin{matrix}
	  \forall \textcolor{teal}{b} \in G,
	  \forall Y \in \Ob(\TOP), \forall s \in \overset{\textcolor{teal}{b}}{\TT} Y, \\
	   \mu(\overset{\textcolor{red}{z}}{\TT}\tau'(\tau(t,s))) = 
	   \mu(\overset{\textcolor{teal}{b}}{\TT}\tau(\tau'(t,s)))
        \end{matrix}
           \right\} ,
	\]
  where, as usual, $\textcolor{red}{z}\in \ZZ(\GG)$ and the topology on
  $\overset{\textcolor{red}{z}}{\ZZ} X$ is the subspace topology inherited from
  the inclusion $\overset{\textcolor{red}{z}}{\ZZ} X \subseteq
  \overset{\textcolor{red}{z}}{\TT } X . $ 
\end{example}

The next example should be no surprise.

\begin{example}
	Every commutative graded monad is naturally isomorphic to its centre.
\end{example}



It was shown in \cite{central-submonads} that not every strong monad is
centralisable. Since strong graded monads are more general, then it should be
clear that not every strong graded monad has a centre. However, just like in
\cite{central-submonads}, the notion of centre for strong graded monads is
ubiquitous and we do not know of any natural counter-examples (i.e.
counter-examples which were not constructed for this sole purpose as in
\cite{central-submonads}).

\section{Central Graded Submonads}
\label{sec:real-central-graded}

We can introduce \emph{central graded submonads} of a strong graded monad in
analogy to the construction in \cite{central-submonads}.

\begin{theorem}[Centrality]\label{th:centrality}
	Let $\ZZ(\GG) = (Z(G),i,\ast)$ be the centre of a pomonoid $\GG = (G,i,\ast)$. 
  Let $\CC$ be a symmetric monoidal category and $\TT$ a strong $\GG$-graded monad on it. Assume that $\TT$ has a centre $\ZZ$ that is graded by $\ZZ(\GG)$.
	Let $\SC$ be a strong $\ZZ(\GG)$-graded submonad of $\TT$ with
	$\iota:\SC \to \TT$ the strong submonad monomorphism.  
	The following are equivalent:
  \begin{enumerate}
    \item[1)] \label{ccond:1} For any element $\textcolor{red}{z}$ in $\ZZ(\GG)$, 
	any object $X$ of $\CC$, we have that $(\overset{\textcolor{red}{z}}{\SC} X,\overset{\textcolor{red}{z}}{\iota}_X)$ is a graded central cone for $\TT$ at $(\textcolor{red}{z},X)$;
    \item[2)] $\SC$ is a commutative graded submonad of the centre of $\TT$ with submonad morphism $\iota^{\SC} \colon \SC \to \ZZ$ such that $\iota \colon \SC \to \TT$ factorises through $\iota^{\SC}.$
  \end{enumerate}
\end{theorem}
\begin{proof}
	$(1\Rightarrow 2).$ 
      Commutativity follows with the same arguments as in Diagram \ref{eq:proof-commutative-graded} (Appendix \ref{sec:app}) by replacing $\ZZ$ with $\mathcal S.$
      Each $\overset{\textcolor{red}{z}}{\iota}_X \colon \overset{\textcolor{red}{z}}{\SC} X \to \overset{\textcolor{red}{z}}{\TT} X$
      factorizes through the terminal central cone $\iota^\ZZ_X$ (by definition).
			These factorisations determine the submonad morphism $\SC \to \ZZ$.
	
	$(2\Rightarrow 1):$ 
      This follows by Lemma~\ref{lem:precompose-central-graded} in Appendix \ref{sec:app}.
	\end{proof}

\begin{definition}[Central Graded Submonad]\label{def:central-sub}
	Given a strong graded submonad $\SC$ of $\TT$, we say that $\SC$ is a \emph{central
	graded submonad} of $\TT$ if it satisfies one of the two equivalent conditions
	from Theorem \ref{th:centrality}.
\end{definition}

Every central graded submonad is commutative and Theorem \ref{th:centrality}
shows that the centre (whenever it exists) can be seen as the largest central
graded submonad of $\TT$.

\section{Relaxations}
\label{sec:lax_commu}

It is unfortunate that the definition of centre, for graded monads, is not able
to make use of the order of the pomonoid. When we started this project, we were
hopping for the inclusion $\ZZ(\GG,\ast)\subseteq \GG$ to use a multiplication
$(\circledast)\neq(\ast)$ which seems reasonable since the inclusion only
requires that $a \circledast b \le a \ast b$. Such a result would mean that the
constructed centre would be able not only to restrict the graded monad to
commutative elements, but also to preserve some non-commutative ones by
approximating their interaction with others. The centre, however, is arguably a
notion that is too universal and which we think prevents us from constructing
such a non-free operation $(\circledast)$.

In this section, we explore perspectives that may allow us to overcome this issue and that make use of the order. 
Intuitively, one of the issues that we faced, was the absence of structure in the pomonoid other than the multiplication and the order.
As a result of this, it is difficult, if not impossible, to canonically construct the new approximated multiplication.
Therefore, we believe that we have to use more structure than that of a pomonoid.
Towards this end, we can see at least three natural extensions:
\begin{itemize}
  \item A natural extension consists in assuming that the commutative
    $(\circledast)$ operation is given in the pomonoid but not for the graded
    monad. We can then try to pull it over the gradation using the central
    construction.
  \item We could consider the order to be much richer, in particular, we could
    require the existence of a quantale, which is, intuitively, a pomonoid over a complete
    lattice. Then we could try to construct limits over all possible
    commutations of elements, generalizing the idea of shuffle. However, this
    is not possible to do in an arbitrary quantale, not even in every quantale over
    free complete lattices because this does not preserve the associativity of
    the new commutative operation.
  \item The last direction is relaxing the notion of commutativity itself,
    hence relaxing the constraints over the centre.
\end{itemize}

\subsection{Bimonoids}

\begin{definition}[A Bimonoid]
  A \emph{bimonoid} is a pomonoid $(G,\le,i,\ast)$ with an additional symmetric monoidal structure $(j,\circledast )$ such that
  \begin{equation}
    a\ast b \le a\circledast b \label{eq:bimonoid}
  \end{equation}
  We use $\delta$ to refer to the above inequality.
\end{definition}

\begin{remark}
The name bimonoid is used for many structures and the above definition is non-standard.
\end{remark}

\begin{property}
  Let $ \GG = (G,\le, i, \ast,j,\circledast )$ be a bimonoid.
  An ordered $(\GG,\le,i,\ast)$-graded monad $\TT$ on a monoidal category $\CC$
  is also an ordered $(\GG,\le,i,\circledast)$-graded monad with the same
  structure except for the multiplication
  $\mu^\circledast_{a,b,X}=\mu^\ast\overset{a\ast b\le a\circledast b}\TT$.
\end{property}

\begin{definition}[Bimonoidal Graded Centre in $\Set$]
  Let $\TT$ be a strong $\GG$-graded monad on $\Set$, so that $\GG$ also has a bimonoidal operation $(\circledast )$.

  For an arbitrary $\textcolor{red}{a}\in \GG$, 
	we say that the $\GG$-graded \emph{centre} of $\TT$ relative to $(\circledast )$ at $(\textcolor{red}{a},X)$, 
	written $\overset{\textcolor{red}{a}}{\ZZ_{\circledast }} X$, is the set 
	\[
	  \overset{\textcolor{red}{a}}{\ZZ_{\circledast }} X \eqdef
          \left\{ t \in \overset{\textcolor{red}{a}}{\TT} X \
          \middle|\
          \begin{matrix}
	    \forall \textcolor{teal}{b} \in G,
	    \forall Y \in \Ob(\Set), \forall s \in \overset{\textcolor{red}{a}}{\TT} Y, \\
	    \overset{a\ast b\le a\circledast b}{\TT}(\mu(\overset{\textcolor{red}{a}}{\TT}\tau'(\tau(t,s)))) = 
	    \overset{b\ast a\le a\circledast b}{\TT}(\mu(\overset{\textcolor{teal}{b}}{\TT}\tau(\tau'(t,s))))
          \end{matrix}
           \right\} .
	\]
	We write $\overset{\textcolor{red}{a}}{\iota}_X : 
	\overset{\textcolor{red}{a}}{\ZZ_{\circledast }} X \subseteq \overset{\textcolor{red}{a}}{\TT} X$ 
	for the indicated subset inclusion.
\end{definition}

The $\GG$-graded centre of $\TT$ relative to $(\circledast)$ is then
commutative on its $(\GG,\circledast)$ gradation but not on its $(\GG,\ast)$
gradation. Since the first gradation is an approximation of the second one, it is suitable for
interpreting programs where we sometimes know of the evaluation order
(e.g. composition) and where we sometimes do not (e.g. binary operators).

\begin{example}
  A simple way to construct an example is to consider a monad graded by a
  pomonoid $\GG$ with an absorbing top element $\top$. In this case we can
  define $a\circledast b$ to be $a\ast b$ when $a,b\in Z(\GG)$ and
  $a\circledast b=\top$ otherwise. This way, the restriction to the $\top$
  grade is also a monad, and we work at the same time with the centre of the
  graded monad when using gradations $a\in Z(\GG)$, the centre of the
  $\overset{\top}\TT$ when we can't compute the grade, and still have access to
  the previous grades when not using $\mu^\circledast$.
\end{example}

This new notion of centre, which can be generalized to other categories, offers
a sightly richer structure since the gradation is not $\ZZ(G)$ but
$(G,\circledast)$.  But it requires knowledge of $\circledast$ and many of the
limitations of the centre still persist.

\subsection{Quantale}
The inspiration for our next idea is as follows: if we had access to sups in the pomonoid, we could try to construct $\circledast$.

A quantale is a monoid in the category of complete lattices and sup-preserving functions. Notice that the tensor product that we have to use on complete lattices is the right adjoint of the arrow $[A\Rightarrow B]$ of sup-preserving functions ordered pointwise. However, this is not an object that is easy to use. Fortunately, on complete lattices freely generated from a poset (the free functor is associating a poset to the complete lattice of its initial segments), it corresponds to the usual tensor on the underlying poset. We are thus focusing on the quantales obtained as the free completions of pomonoids.

Under some reasonable conditions (e.g. $\GG$ has all $\kappa$-colimits, and $\overset{\textcolor{red}{a}}{\TT}$ are $\kappa$-ary functors) a monad graded by a pomonoid $\GG$ seems to be definable as a monad graded by the quantale $\mathcal{I}(\GG)$ by computing the left Kan extension of $\TT:\GG\rightarrow[\CC,\CC]$ along the inclusion $\GG\rightarrow \mathcal{I}(\GG)$ in \texttt{MonCat}.\footnote{The difficulty here is that doing this in \texttt{MonCat} is much trickier than doing it in \texttt{Cat}.}

Thus finding a commutative over-approximant $\circledast$ of $(\ast)$ in $\mathcal{I}(\GG)$ will permit, via the above construction, to obtain a commutative graded monad. The question is how to obtain $\circledast$. 

A natural candidate is $a\circledast b \eqdef (a\ast b)\vee(b\ast a)$, but, unfortunately, it is generally not an associative operation. A more reasonable approach consists in computing the shuffle operation, or rather a generalization of it. In this case most examples generate an associative operation, but not all. A free construction is yet to be discovered for a good over-approximation.

\subsection{Duoids and Concurrent Monads}

In order to go in this last direction, we first need to investigate how to relax the notion of commutativity.
Indeed, there is no way to change the monadic composition along this new operation over grades, but we can ``approximate'' the composition along this operation.

Commutativity, in its standard form, is
difficult to relax naturally. However, one can present commutativity of a monad
through its monoidality.

\begin{lemma}
  A monad $\TT$ in a monoidal category $(\CC,\otimes,e)$ is commutative iff 
  there exists a natural transformation $m_{X,Y} : \TT X\otimes\TT Y\rightarrow \TT(X\otimes Y)$ such that :
  \begin{center}
  \begin{tikzpicture}
    \node (TT) at (0,2) {$\TT\TT X\otimes\TT\TT Y$};
    \node (DT) at (6,2) {$\TT (\TT X\otimes\TT Y)$};
    \node (DD) at (12,2) {$\TT\TT (X\otimes Y)$};
    \node (D) at (12,0) {$\TT (X\otimes Y)$};
    \node (T) at (0,0) {$\TT X\otimes\TT Y$};
    \draw[->] (TT) to node [auto] {{\scriptsize $m$}} (DT);
    \draw[->] (DT) to node [auto] {{\scriptsize $\TT m$}} (DD);
    \draw[->] (DD) to node [auto] {{\scriptsize $\mu$}} (D);
    \draw[->] (TT) to node [auto] {{\scriptsize $\mu\otimes\mu$}} (T);
    \draw[->] (T) to node [auto] {{\scriptsize $m$}} (D);
  \end{tikzpicture}
  \end{center}
  and
  $$(\eta\otimes\eta);m=\eta \quad (m\otimes\id);m;\TT\alpha=\alpha(\id\otimes m);m $$
  $$(\eta\otimes\id);m;\TT\lambda=\lambda \quad (\id\otimes\eta);m;\TT\rho=\rho$$
\end{lemma}
  
This definition can be relaxed by orienting the main diagram.
\begin{definition}[Lax Commutative]
  A monad $\TT$ in an order-enriched monoidal category $(\CC,\otimes,e)$ is said to be lax commutative iff there exists a natural transformation $m_{X,Y} : \TT X\otimes\TT Y\rightarrow \TT(X\otimes Y)$ such that :
  \begin{center}
    \begin{tikzpicture}
      \node (TT) at (0,2) {$\TT\TT X\otimes\TT\TT Y$};
      \node (DT) at (6,2) {$\TT (\TT X\otimes\TT Y)$};
      \node (DD) at (12,2) {$\TT\TT (X\otimes Y)$};
      \node (D) at (12,0) {$\TT (X\otimes Y)$};
      \node (T) at (0,0) {$\TT X\otimes\TT Y$};
      \node () at (6,1) {$\Downarrow$};
      \draw[->] (TT) to node [auto] {{\scriptsize $m$}} (DT);
      \draw[->] (DT) to node [auto] {{\scriptsize $\TT m$}} (DD);
      \draw[->] (DD) to node [auto] {{\scriptsize $\mu$}} (D);
      \draw[->] (TT) to node [auto] {{\scriptsize $\mu\otimes\mu$}} (T);
      \draw[->] (T) to node [auto] {{\scriptsize $m$}} (D);
    \end{tikzpicture}
  \end{center}
         and
         $$(\eta\otimes\eta);m=\eta \quad (m\otimes\id);m;\TT\alpha=\alpha(\id\otimes m);m $$
         $$(\eta\otimes\id);m;\TT\lambda=\lambda \quad (\id\otimes\eta);m;\TT\rho=\rho$$
\end{definition}

Such a definition means that the monad is not commutative, but can be over-approximated as such. In terms of proper programs, it means that the monad is not commutative but it has been extended with a kind of non-determinism so that one can approximate the operator effects by considering all their eventual behaviors (left-right, right-left, but also interleaving or parallelism).

This concept is closely related with that of a concurrent monad, a recent 
concept which is basically a lax-commutative monad with an additional relaxation 
on the unit. In this work, we do not dwell on concurrent monads, no more than we 
indulge in bicategories, but we do use the algebraic object of which the 
former is a categorification: the duoids.

\begin{definition}[Duoid]
  A duoid is a pomonoid $(G,\le,i,\ast)$ with an additional symmetric monoidal structure $(j,\parallel )$ such that
  \begin{equation}
    (a\parallel c)\ast(b\parallel d) \le (a\ast b)\parallel (c\ast d)   \label{eq:duoid}
  \end{equation}
  We call $\delta$ the first inequality.
\end{definition}
One can see this as a poset $(G,\le)$ with two multiplicative operations $(\ast)$ and $(\parallel)$ making it a pomonoid in two different ways. 
The first half is non-commutative and represent the sequential composition, the second half is commutative and represents the parallel composition. 
The parallel composition can be seen as much less precise than the sequential composition. We can deduce from~\eqref{eq:duoid} that $a\ast b\le a\parallel b$, i.e. $(\parallel)$ is approximating $(\ast)$. Therefore, it can play the rôle of $(\circledast)$ in the introduction.
\begin{remark}
	The name duoid is used for many structures. Some non-equivalent definitions may or may not use another unit $i \le j$ for the operation $(\parallel )$, they may or may not require symmetry/braiding of $(\parallel )$ and/or $(\ast)$, they also can be strict/weak/colax. It is important for us to use exactly the above definition.
\end{remark}

\begin{definition}[Duoidal Gradation]
  Let $ \GG : (G,\le, i, \ast,j,\parallel )$ be a duoid. 
	A duoidal $\GG$-graded monad on a monoidal category $\CC$ is given by the following data:
        \begin{itemize}
        \item an ordered $(G,\le, i, \ast)$-graded monad $\TT$
        \item a transformation $m_{a,b,X,Y} : \overset{a}{\TT} X\otimes\overset{b }{\TT} Y\rightarrow \overset{a\parallel b}{\TT}(X\otimes Y)$ natural in $a$, $b$, $X$ and $Y$.
        \end{itemize}
	such that the following diagrams commute:
        \begin{center}
          \begin{tikzpicture}
            \node (TT) at (0,4) {$\overset{a}\TT\overset{b}\TT X\otimes\overset{c}\TT\overset{d}\TT Y$};
            \node (DT) at (8,4) {$\overset{a\parallel c}\TT (\overset{b}\TT X\otimes\overset{d}\TT Y)$};
            \node (DD) at (16,4) {$\overset{a\parallel c}\TT\ \overset{b\parallel d}\TT (X\otimes Y)$};
            \node (D1) at (8,0) {$\overset{(a\parallel c)\ast(b\parallel d)}\TT (X\otimes Y)$};
            \node (D2) at (16,1) {$\overset{(a\ast b)\parallel (c\ast d)}\TT (X\otimes Y)$};
            \node (T) at (0,0) {$\overset{a\ast b}\TT X\otimes\overset{c\ast d}\TT Y$};
            \draw[->] (TT) to node [auto] {{\scriptsize $m$}} (DT);
            \draw[->] (DT) to node [auto] {{\scriptsize $\TT m$}} (DD);
            \draw[->] (DD) to node [auto] {{\scriptsize $\mu$}} (D2);
            \draw[->] (TT) to node [auto] {{\scriptsize $\mu\otimes\mu$}} (T);
            \draw[->] (T) to node [auto] {{\scriptsize $m$}} (D1);
            \draw[->] (D2) to node [auto] {{\scriptsize $\overset{\textcolor{red}{\delta}}\TT$}} (D1);
          \end{tikzpicture}
        \end{center}
        as well as the other monoidal diagrams:
         $$(\eta\otimes\eta);m=\eta \quad (m\otimes\id);m;\TT\alpha=\alpha(\id\otimes m);m $$
         $$(\eta\otimes\id);m;\TT\lambda=\lambda \quad (\id\otimes\eta);m;\TT\rho=\rho$$
\end{definition}

\begin{example}
  Let $\Sigma$ be an alphabet, then $\mathcal{P}(\Sigma^*)$ is a duoid with the concatenation and shuffle operations
  $$ L\ast L' \eqdef \{ww'\mid w\in L, w'\in L'\} \quad i=\{\epsilon\}$$
  $$ L\parallel L' \eqdef \{w_1w'_1...w_nw'_n\mid w_1...w_n\in L, w'_1...w'_n\in L'\} \quad j=\epsilon$$
  This duoid is grading the corresponding writer monad:
  $$\overset{L}{\TT}X \eqdef \{(x,L')\mid x\in X, L'\subseteq L\}) \quad \overset{L}{\TT}f(x,L') \eqdef (f(x),L')$$
  $$\overset{L\subseteq L'}{\TT}X \texttt{ is the inclusion of sets} $$
  $$ \eta(x)=(x,i) \quad\quad \mu((x,L),L')=(x,L\ast L') $$
  $$ m((x,L),(x',L'))= ((x,x'),L\parallel L') $$
\end{example}

Importing the construction inspired from the centre in such a duoidal framework may result in some more refined results. But it is still unclear (to us) how the notion of central cone can be remodeled to fit the monoidality diagram and the notion of lax commutativity.

\section*{Conclusion and Future Work}
\label{sec:conclude}
In this work, we showed how to construct the centre of graded monads
and how to formulate central graded submonads, 
where the gradation is given by pomonoids and where the monads are defined over symmetric monoidal categories. 
We also introduce the lax commutativity on graded monads by choosing duoids as gradations, 
to make use of different pomonoid structures.

As part of future work, it will be interesting to see if there are other ways to take the order of a pomonoid into account when constructing the centre and central cones. Furthermore, we have not provided any suitable universal conditions (similar to the ones in \cite{central-submonads}) that characterise the centre of a graded monad, so this is another open problem.

\section*{Acknowledgments}
We thank Louis Lemonnier, Dylan McDermott and Tarmo Uustalu for discussions related to the
paper.  QL gratefully acknowledges financial support from LIPN which funded his
internship during which a lot of the work was done.

\newpage

%
%
%
\bibliographystyle{splncs04}
%

\bibliography{refs}

@article{commutants,
	author    = {Rory B. B. Lucyshyn{-}Wright},
	title     = {Commutants for Enriched Algebraic Theories and Monads},
	journal   = {Appl. Categorical Struct.},
	volume    = {26},
	number    = {3},
	pages     = {559--596},
	year      = {2018},
	url_commented       = {https://doi.org/10.1007/s10485-017-9503-1},
	doi       = {10.1007/s10485-017-9503-1},
	biburl_commented    = {https://dblp.org/rec/journals/acs/Lucyshyn-Wright18a.bib},
	bibsource = {dblp computer science bibliography, https://dblp.org}
}

@article{premonoidal,
  title={Premonoidal Categories and Notions of Computation},
  author={John Power and Edmund P. Robinson},
  journal={Math. Struct. Comput. Sci.},
  year={1997},
  volume={7},
  pages={453-468}
}

@article{commutativity,
title = {Commutativity},
journal = {Journal of Pure and Applied Algebra},
volume = {220},
number = {5},
pages = {1707-1751},
year = {2016},
issn = {0022-4049},
doi = {10.1016/j.jpaa.2015.09.003},
url_commented = {https://www.sciencedirect.com/science/article/pii/S0022404915002510},
author = {Richard Garner and Ignacio {López Franco}}
}

@inproceedings{central-submonads,
  author       = {Titouan Carette and
                  Louis Lemonnier and
                  Vladimir Zamdzhiev},
  title        = {Central Submonads and Notions of Computation: Soundness, Completeness
                  and Internal Languages},
  booktitle    = {{LICS}},
  pages        = {1--13},
  year         = {2023},
  url_commented          = {https://doi.org/10.1109/LICS56636.2023.10175687},
  doi          = {10.1109/LICS56636.2023.10175687},
  biburl_commented       = {https://dblp.org/rec/conf/lics/CaretteLZ23.bib},
  bibsource    = {dblp computer science bibliography, https://dblp.org}
}

@inproceedings{FujiiKM16,
  author       = {Soichiro Fujii and
                  Shin{-}ya Katsumata and
                  Paul{-}Andr{\'{e}} Melli{\`{e}}s},
  editor       = {Bart Jacobs and
                  Christof L{\"{o}}ding},
  title        = {Towards a Formal Theory of Graded Monads},
  booktitle    = {Foundations of Software Science and Computation Structures - 19th
                  International Conference, {FOSSACS} 2016, Held as Part of the European
                  Joint Conferences on Theory and Practice of Software, {ETAPS} 2016,
                  Eindhoven, The Netherlands, April 2-8, 2016, Proceedings},
  series       = {Lecture Notes in Computer Science},
  volume       = {9634},
  pages        = {513--530},
  publisher    = {Springer},
  year         = {2016},
  url_commented          = {https://doi.org/10.1007/978-3-662-49630-5\_30},
  doi          = {10.1007/978-3-662-49630-5\_30},
  biburl_commented       = {https://dblp.org/rec/conf/fossacs/FujiiKM16.bib},
  bibsource    = {dblp computer science bibliography, https://dblp.org}
}

@inproceedings{IMO20,
  author       = {Andrej Ivaskovic and
                  Alan Mycroft and
                  Dominic Orchard},
  editor       = {Zena M. Ariola},
  title        = {Data-Flow Analyses as Effects and Graded Monads},
  booktitle    = {5th International Conference on Formal Structures for Computation
                  and Deduction, {FSCD} 2020, June 29-July 6, 2020, Paris, France (Virtual
                  Conference)},
  series       = {LIPIcs},
  volume       = {167},
  pages        = {15:1--15:23},
  publisher    = {Schloss Dagstuhl - Leibniz-Zentrum f{\"{u}}r Informatik},
  year         = {2020},
  url_commented          = {https://doi.org/10.4230/LIPIcs.FSCD.2020.15},
  doi          = {10.4230/LIPICS.FSCD.2020.15},
  biburl_commented       = {https://dblp.org/rec/conf/fscd/IvaskovicMO20.bib},
  bibsource    = {dblp computer science bibliography, https://dblp.org}
}

@inproceedings{OWE20,
  author       = {Dominic Orchard and
                  Philip Wadler and
                  Harley {Eades III}},
  editor       = {Max S. New and
                  Sam Lindley},
  title        = {Unifying graded and parameterised monads},
  booktitle    = {Proceedings Eighth Workshop on Mathematically Structured Functional
                  Programming, MSFP@ETAPS 2020, Dublin, Ireland, 25th April 2020},
  series       = {{EPTCS}},
  volume       = {317},
  pages        = {18--38},
  year         = {2020},
  url_commented          = {https://doi.org/10.4204/EPTCS.317.2},
  doi          = {10.4204/EPTCS.317.2},
  biburl_commented       = {https://dblp.org/rec/journals/corr/abs-2001-10274.bib},
  bibsource    = {dblp computer science bibliography, https://dblp.org}
}

@inproceedings{canonical,
  author       = {Flavien Breuvart and
                  Dylan McDermott and
                  Tarmo Uustalu},
  editor       = {Jade Master and
                  Martha Lewis},
  title        = {Canonical Gradings of Monads},
  booktitle    = {Proceedings Fifth International Conference on Applied Category Theory,
                  {ACT} 2022, Glasgow, United Kingdom, 18-22 July 2022},
  series       = {{EPTCS}},
  volume       = {380},
  pages        = {1--21},
  year         = {2022},
  url_commented          = {https://doi.org/10.4204/EPTCS.380.1},
  doi          = {10.4204/EPTCS.380.1},
  biburl_commented       = {https://dblp.org/rec/journals/corr/abs-2307-16558.bib},
  bibsource    = {dblp computer science bibliography, https://dblp.org}
}

@unpublished{commu-graded-22,
	title={Commutative graded monads},
	author={Rowan Poklewski-Koziell},
  NOTE = {working paper or preprint},
  url_commented={https://arxiv.org/abs/2204.01634},
	year={2022}
}

@article{perf_pomo,
author = {Gould, Victoria and Shaheen, Lubna},
year = {2010},
month = {08},
pages = {102-127},
title = {Perfection for pomonoids},
volume = {81},
journal = {Semigroup Forum},
doi = {10.1007/s00233-010-9237-y}
}

\appendix

\section{Coherence Properties of costrength}
\label{sec:app}

\begin{proposition}[Coherence Properties of costrength]\label{prop:costrength-graded}
	For all elements $\textcolor{red}{z}, \textcolor{teal}{b}$ in $G$, $X,Y$ in $\CC$, 
	the following diagrams commute:
	\[\stikz[0.83]{../figures/costrength-graded.tikz}\]
\end{proposition}

\begin{proof}
	Only $\eta$ and $\mu$ part of this proposition is going to be used in the proofs later, 
	hence we only give proofs on those two parts, the rest could be proved similarly. 

	The proof of $\eta$ :
	\[\stikz[0.83]{../figures/costrength-graded-eta.tikz}\]
	(1) $\gamma$ is natural; (2) $\gamma_{X,Y} ; \gamma_{Y,X} = id$; (3) definition of $\tau$; 
	(4) definition of $\tau'$; and (5) $\eta$ and $\gamma$ are natural.

	$\mu$ :
	\[\stikz[0.80]{../figures/costrength-graded-mu.tikz}\]
	(1) fact that $\overset{\textcolor{red}{a}}{\TT} \gamma_{ \TT X ,  Y} = 
	\overset{\textcolor{red}{a}}{\TT} \gamma_{ Y, \TT X }^{-1}$, and definition of $\tau'$;
	(2) $\overset{\textcolor{red}{a}}{\TT}$ is a functor and definition of $\tau'$; 
	(3) $\gamma$ is natural; (4) definition of strength; 
	(5) $\mu_{X\otimes Y}^{\textcolor{red}{a} , \textcolor{teal}{b}}$ and $\gamma$ are natural and 
	(6) definition of $\tau'$.
\end{proof}

\section{Lemmas for proof of Theorem \ref{th:submonad-from-cone-graded}}
\label{app:proofs}

\begin{lemma}\label{lem:precompose-central-graded}
	If $(X,f:X\to \overset{ \textcolor{red}{z}}{\TT} Y)$ is a graded central cone of $\TT$ at $(\textcolor{red}{z},Y)$.
	Then for any $g:Z\to X$ in $\CC$, it follows that $(Z,f\circ g)$ is a graded central cone of $\TT$ at 
	$(\textcolor{red}{z},Y)$.
\end{lemma}

\begin{proof} 
	This is obtained by precomposing the definition of graded central cone by $g\otimes \id$. 
	For all $\textcolor{teal}{b}\in G$ and $X$ in $\CC$,
  \[\scalebox{0.8}{\begin{tikzpicture}
	\begin{pgfonlayer}{nodelayer}
		\node [style=empty] (0) at (-6, 3) {$X\otimes \overset{\textcolor{teal}{b}}{\TT}  X$};
		\node [style=empty] (1) at (4, 3) {$\overset{\textcolor{red}{z}}{\TT}Y\otimes \overset{\textcolor{teal}{b}}{\TT}  X$};
		\node [style=empty] (2) at (14, 3) {$\overset{\textcolor{red}{z}}{\TT}(Y\otimes \overset{\textcolor{teal}{b}}{\TT} X)$};
		\node [style=empty] (3) at (14, 0) {$\overset{\textcolor{red}{z}}{\TT}  \overset{\textcolor{teal}{b}}{\TT} (Y\otimes X)$};
		\node [style=empty] (4) at (7.5, -3) {$\overset{\textcolor{teal}{b} \ast \textcolor{red}{z}}{\TT} (Y\otimes X)$};
		\node [style=empty] (5) at (-6, 0) {$\overset{\textcolor{red}{z}}{\TT}Y\otimes\overset{\textcolor{teal}{b}}{\TT} X$};
		\node [style=empty] (6) at (-6, -3) {$ \overset{\textcolor{teal}{b}}{\TT} (\overset{\textcolor{red}{z}}{\TT} Y\otimes X)$};
		\node [style=empty] (7) at (0.5, -3) {$ \overset{\textcolor{teal}{b}}{\TT}  \cdot \overset{\textcolor{red}{z}}{\TT} (Y\otimes X)$};
		\node [style=empty] (8) at (-6, 2.5) {};
		\node [style=empty] (9) at (-6, 0.5) {};
		\node [style=empty] (10) at (-6, -0.5) {};
		\node [style=empty] (11) at (-6, -2.5) {};
		\node [style=empty] (12) at (14, 2.5) {};
		\node [style=empty] (13) at (14, 0.5) {};
		\node [style=empty] (14) at (14, -0.5) {};
		\node [style=empty] (15) at (14, -2.5) {};
		\node [style=empty] (16) at (-1, 3.75) {$f\otimes \overset{\textcolor{teal}{b}}{\TT}  X$};
		\node [style=empty] (17) at (9, 3.75) {$\tau'$};
		\node [style=empty] (18) at (15.5, 1.5) {$\overset{\textcolor{red}{z}}{\TT}\tau_{Y,X}$};
		\node [style=empty] (19) at (15.5, -1.5) {$\mu_{Y\otimes X}^{\textcolor{red}{z} ,  \textcolor{teal}{b}}$};
		\node [style=empty] (20) at (-7.5, 1.5) {$f\otimes\overset{\textcolor{teal}{b}}{\TT} X$};
		\node [style=empty] (21) at (-7.5, -1.5) {$\tau$};
		\node [style=empty] (22) at (-2.75, -4) {$ \overset{\textcolor{teal}{b}}{\TT} \tau'_{Y,X}$};
		\node [style=empty] (23) at (4, -4) {$\mu_{Y\otimes X}^{\textcolor{teal}{b} , \textcolor{red}{z}}$};
		\node [style=empty] (24) at (-12.25, 3) {$\overset{\textcolor{red}{z}}{Z}\otimes \overset{\textcolor{teal}{b}}{\TT} X$};
		\node [style=empty] (27) at (-9, 3.75) {$g\otimes \overset{\textcolor{teal}{b}}{\TT}  X$};
		\node [style=empty] (28) at (14, -3) {$\overset{\textcolor{red}{z} \ast \textcolor{teal}{b}}{\TT} (Y\otimes X)$};
	\end{pgfonlayer}
	\begin{pgfonlayer}{edgelayer}
		\draw [style=to] (0) to (1);
		\draw [style=to] (1) to (2);
		\draw [style=to] (6) to (7);
		\draw [style=to] (7) to (4);
		\draw [style=to] (8) to (9);
		\draw [style=to] (10) to (11);
		\draw [style=to] (14) to (15);
		\draw [style=to] (12) to (13);
		\draw [style=to] (24) to (0);
		\draw [style=equal-arrow, in=180, out=0] (4) to (28);
	\end{pgfonlayer}
\end{tikzpicture}} \]
  commutes directly from the definition of graded central cone for $f$.
\end{proof}

\begin{lemma}\label{lem:postcompose-central-graded}
	If $(X,f:X\to \overset{ \textcolor{red}{z}}{\TT} Y)$ is a graded central cone of $\TT$ at $(\textcolor{red}{z},Y)$.
	Then for any $g:Y\to Z$ in $\CC$ at $\textcolor{red}{z}$, 
	it follows that $(X,\overset{ \textcolor{red}{z}}{\TT} g\circ f)$ is a graded central cone of $\TT$ at $Z$ at $\textcolor{red}{z}$.
\end{lemma}
\begin{proof}
	The naturality of $\tau$ and $\mu$ allow us to push the application of $g$ to the
	last postcomposition, in order to use the central property of $f$. 
	In more details, for all $\textcolor{teal}{b}\in G$ and $X$ in $\CC$,
  	the following diagram:
	\[\scalebox{0.8}{\begin{tikzpicture}
	\begin{pgfonlayer}{nodelayer}
		\node [style=empty] (0) at (-13, 13) {$X\otimes \overset{\textcolor{teal}{b}}{\TT} X$};
		\node [style=empty] (1) at (15, 13) {$\overset{\textcolor{red}{z}}{\TT} Y\otimes\overset{\textcolor{teal}{b}}{\TT} X$};
		\node [style=empty] (2) at (24, 13) {$\overset{\textcolor{red}{z}}{\TT} Z\otimes\overset{\textcolor{teal}{b}}{\TT} X$};
		\node [style=empty] (3) at (24, 3) {$\overset{\textcolor{red}{z}}{\TT}  \overset{\textcolor{teal}{b}}{\TT} (Z\otimes X)$};
		\node [style=empty] (4) at (24, -3.25) {$\overset{\textcolor{red}{z} \ast \textcolor{teal}{b} }{\TT}  (Z\otimes X)$};
		\node [style=empty] (5) at (-13, -3.25) {$\overset{\textcolor{red}{z}}{\TT} Z\otimes \overset{\textcolor{teal}{b}}{\TT}  X$};
		\node [style=empty] (6) at (-6, -3.25) {$\overset{\textcolor{teal}{b}}{\TT}(\overset{\textcolor{red}{z}}{\TT}Z\otimes X)$};
		\node [style=empty] (7) at (1, -3.25) {$\overset{\textcolor{teal}{b}}{\TT}  \overset{\textcolor{red}{z}}{\TT}  (Z\otimes X)$};
		\node [style=empty] (8) at (-13, 12.5) {};
		\node [style=empty] (9) at (-13, 0.75) {};
		\node [style=empty] (10) at (-13, -0.25) {};
		\node [style=empty] (11) at (-13, -2.75) {};
		\node [style=empty] (12) at (24, 7.5) {};
		\node [style=empty] (13) at (24, 3.5) {};
		\node [style=empty] (14) at (24, 2.5) {};
		\node [style=empty] (15) at (24, -2.75) {};
		\node [style=empty] (16) at (1, 13.75) {$f\otimes\overset{\textcolor{teal}{b}}{\TT} X$};
		\node [style=empty] (17) at (25.25, 10.75) {$\tau'$};
		\node [style=empty] (18) at (25.25, 5.5) {$\overset{\textcolor{red}{z}}{\TT}\tau_{Z,X}$};
		\node [style=empty] (19) at (25.25, 0) {$\mu_{Z\otimes X}^{\textcolor{red}{z}, \textcolor{teal}{b}}$};
		\node [style=empty] (20) at (-14.5, 7) {$f\otimes\overset{\textcolor{teal}{b}}{\TT} X$};
		\node [style=empty] (21) at (-9.5, -4) {$\tau$};
		\node [style=empty] (22) at (-2.25, -4) {$\overset{\textcolor{teal}{b}}{\TT}\tau'_{Z,X}$};
		\node [style=empty] (23) at (5.5, -4) {$\mu_{Z\otimes X}^{\textcolor{teal}{b}, \textcolor{red}{z}}$};
		\node [style=empty] (24) at (24, 12.5) {};
		\node [style=empty] (25) at (24, 8.5) {};
		\node [style=empty] (26) at (24, 8) {$\overset{\textcolor{red}{z}}{\TT} (Z\otimes\overset{\textcolor{teal}{b}}{\TT} X)$};
		\node [style=empty] (27) at (-13, 0.25) {$\overset{\textcolor{red}{z}}{\TT} Y\otimes \overset{\textcolor{teal}{b}}{\TT} X$};
		\node [style=empty] (28) at (-14.75, -1.5) {$\overset{\textcolor{red}{z}}{\TT} g\otimes \overset{\textcolor{teal}{b}}{\TT}  X$};
		\node [style=empty] (29) at (19.5, 13.75) {$\overset{\textcolor{red}{z}}{\TT} g\otimes\overset{\textcolor{teal}{b}}{\TT} X$};
		\node [style=empty] (30) at (15, 8) {$\overset{\textcolor{red}{z}}{\TT} (Y\otimes\overset{\textcolor{teal}{b}}{\TT} X)$};
		\node [style=empty] (31) at (15, 3) {$\overset{\textcolor{red}{z}}{\TT}  \overset{\textcolor{teal}{b}}{\TT} (Y\otimes X)$};
		\node [style=empty] (32) at (15, 12.5) {};
		\node [style=empty] (33) at (15, 8.5) {};
		\node [style=empty] (34) at (15, 7.5) {};
		\node [style=empty] (35) at (15, 3.5) {};
		\node [style=empty] (36) at (15, 0.25) {$\overset{\textcolor{red}{z} \ast \textcolor{teal}{b} }{\TT}  (Y\otimes X)$};
		\node [style=empty] (37) at (16, 0) {};
		\node [style=empty] (38) at (23, -2.75) {};
		\node [style=empty] (39) at (15, 2.5) {};
		\node [style=empty] (40) at (15, 0.75) {};
		\node [style=empty] (41) at (13.75, 10.75) {$\tau'$};
		\node [style=empty] (42) at (13.75, 5.5) {$\overset{\textcolor{red}{z}}{\TT}\tau_{Y,X}$};
		\node [style=empty] (43) at (13.75, 1.75) {$\mu_{Y\otimes X}^{\textcolor{red}{z}, \textcolor{teal}{b}}$};
		\node [style=empty] (44) at (19.5, 9) {$\overset{\textcolor{red}{z}}{\TT}(g\otimes \overset{\textcolor{teal}{b}}{\TT}  X)$};
		\node [style=empty] (45) at (19.5, 3.75) {$\overset{\textcolor{red}{z}}{\TT}  \overset{\textcolor{teal}{b}}{\TT} ( g\otimes X)$};
		\node [style=empty] (46) at (20, -0.5) {$\overset{\textcolor{red}{z} \ast \textcolor{teal}{b} }{\TT}  (g\otimes X)$};
		\node [style=empty] (47) at (-6, 0.25) {$\overset{\textcolor{teal}{b}}{\TT}(\overset{\textcolor{red}{z}}{\TT}Y\otimes X)$};
		\node [style=empty] (48) at (1, 0.25) {$\overset{\textcolor{teal}{b}}{\TT}  \overset{\textcolor{red}{z}}{\TT} (Y\otimes X)$};
		\node [style=empty] (49) at (1, -0.25) {};
		\node [style=empty] (50) at (1, -2.75) {};
		\node [style=empty] (51) at (-6, -0.25) {};
		\node [style=empty] (52) at (-6, -2.75) {};
		\node [style=empty] (53) at (-9.5, 1) {$\tau$};
		\node [style=empty] (54) at (-2.5, 1) {$\overset{\textcolor{teal}{b}}{\TT}\tau'_{Y,X}$};
		\node [style=empty] (55) at (4.5, 1) {$\mu_{Y\otimes X}^{\textcolor{teal}{b},  \textcolor{red}{z}}$};
		\node [style=empty] (56) at (-8, -1.5) {$\overset{\textcolor{teal}{b}}{\TT} (\overset{\textcolor{red}{z}}{\TT}g\otimes X)$};
		\node [style=empty] (57) at (3.5, -1.5) {$\overset{\textcolor{teal}{b}}{\TT}  \overset{\textcolor{red}{z}}{\TT} ( g\otimes X)$};
		\node [style=none] (58) at (1, 7) {\textcolor{blue}{(1)}};
		\node [style=none] (59) at (19.5, 10.5) {\textcolor{blue}{(2)}};
		\node [style=none] (60) at (19.5, 5.75) {\textcolor{blue}{(3)}};
		\node [style=none] (61) at (19.5, 1.5) {\textcolor{blue}{(4)}};
		\node [style=none] (62) at (-10.5, -1.5) {\textcolor{blue}{(5)}};
		\node [style=none] (63) at (-2.5, -1.5) {\textcolor{blue}{(6)}};
		\node [style=none] (64) at (6.5, -1.75) {\textcolor{blue}{(7)}};
		\node [style=empty] (65) at (8, 0.25) {$\overset{\textcolor{teal}{b} \ast \textcolor{red}{z} }{\TT}  (Y\otimes X)$};
		\node [style=empty] (66) at (12, -3.25) {$\overset{\textcolor{teal}{b} \ast \textcolor{red}{z} }{\TT}  (Z\otimes X)$};
		\node [style=empty] (67) at (8, -0.25) {};
		\node [style=empty] (68) at (12, -2.75) {};
		\node [style=empty] (69) at (12.25, -1.25) {$\overset{\textcolor{teal}{b} \ast  \textcolor{red}{z} }{\TT}  (g\otimes X)$};
		\node [style=none] (70) at (15, -1.75) {\textcolor{blue}{(8)}};
	\end{pgfonlayer}
	\begin{pgfonlayer}{edgelayer}
		\draw [style=to] (0) to (1);
		\draw [style=to] (1) to (2);
		\draw [style=to] (6) to (7);
		\draw [style=to] (8) to (9);
		\draw [style=to] (10) to (11);
		\draw [style=to] (14) to (15);
		\draw [style=to] (12) to (13);
		\draw (5) to (6);
		\draw [style=to] (24) to (25);
		\draw [style=to] (30) to (26);
		\draw [style=to] (32) to (33);
		\draw [style=to] (34) to (35);
		\draw [style=to] (31) to (3);
		\draw [style=to] (37) to (38);
		\draw [style=to] (39) to (40);
		\draw [style=to] (27) to (47);
		\draw [style=to] (47) to (48);
		\draw [style=to] (51) to (52);
		\draw [style=to] (49) to (50);
		\draw [style=equal-arrow] (65) to (36);
		\draw [style=equal-arrow] (66) to (4);
		\draw [style=to] (48) to (65);
		\draw [style=to] (67) to (68);
		\draw [style=to] (7) to (66);
	\end{pgfonlayer}
\end{tikzpicture}}\]
  	commutes, because: (1) $f$ is a graded central cone, (2) $\tau'$ is natural, (3) $\tau$ is natural, (4) $\mu$ is natural
  	(5) $\tau$ is natural, (6) $\tau'$ is natural, (7) $\mu$ is natural, (8) $\TT$ is a functor.
\end{proof}

\begin{lemma}\label{lem:monic}
  If $(Z,\iota)$ is a terminal (graded) central cone of $\TT$ at $X$, then $\iota$ is a monomorphism.
\end{lemma}
\begin{proof}
  Let us consider $f,g:Y\to Z$ such that $\iota\circ f=\iota\circ g$; this family of
  morphism is a graded central cone at $X$ (Lemma~\ref{lem:precompose-central-graded}), and
  since $(Z,\iota)$ is a terminal graded central cone, it factors uniquely through
  $\iota$. Thus $f = g$ and therefore $\iota$ is monic.
\end{proof}

\begin{lemma}\label{lem:tau}
	For $ A := (W\otimes \TT X )\otimes Y $ 
	\[ \tau'_{W\otimes X, Y } \circ \tau_{W,X} \otimes Y \circ A = 
		\TT \alpha^{-1}_{W,X,Y} \circ \tau_{W, X\otimes Y } \circ 
		W\otimes \tau'_{X, Y} \circ \alpha_{W,\TT X, Y} \circ A \] 
\end{lemma}

\begin{proof}
	Left  $ = \TT (W\otimes X) \otimes Y $ = Right.
\end{proof}

\section{Proof of Theorem \ref{th:submonad-from-cone-graded}}

\begin{proof}[Proof of Theorem~\ref{th:submonad-from-cone-graded}]
	
	This proof has 3 parts.

	First we prove that $\ZZ$ is a functor;
	then we give its graded monad structure and prove it,
	by showing all of its morphisms exist and are unique, and also natural;
	last we show that it's a strong and commutative graded monad.

	First part:
	
	The first part is following same proof strategy as that of the Theorem in \cite{central-submonads},
	but on graded monads, which are lax monoidal functors between gradation and endofunctor category of $\CC$.

	Recall that $\overset{\textcolor{red}{z}}{\ZZ}$ maps every object $X$ to its terminal central cone at $(\textcolor{red}{z},X)$.
	Let $f:X\to Y$ be a morphism. 
	$\overset{ \textcolor{red}{z}}{\TT} f\circ\overset{\textcolor{red}{z}}{\iota}_X:\overset{ \textcolor{red}{z}}{\ZZ} X\to \overset{ \textcolor{red}{z}}{\TT} Y$ is a central cone according to Lemma~\ref{lem:postcompose-central-graded}.
	Therefore, by proposition \ref{prop:unique-graded}, 
	we can define $\overset{\textcolor{red}{z}}{\ZZ} f$ as the unique map such that the following diagram commutes:
	\[\begin{tikzpicture}
	\begin{pgfonlayer}{nodelayer}
		\node [style=empty] (0) at (-1.5, 1.5) {$\overset{\textcolor{red}{z}}{\ZZ}X$};
		\node [style=empty] (1) at (1.5, 1.5) {$\overset{\textcolor{red}{z}}{\ZZ} Y$};
		\node [style=empty] (2) at (-1.5, -1.5) {$\overset{\textcolor{red}{z}}{\ZZ} X$};
		\node [style=empty] (3) at (1.5, -1.5) {$\overset{\textcolor{red}{z}}{\ZZ} Y$};
		\node [style=empty] (4) at (-1.5, 1) {};
		\node [style=empty] (5) at (1.5, 1) {};
		\node [style=empty] (6) at (1.5, -1) {};
		\node [style=empty] (7) at (-1.5, -1) {};
		\node [style=empty] (8) at (-2.25, 0) {$\iota_X$};
		\node [style=empty] (9) at (2.25, 0) {$\iota_Y$};
		\node [style=empty] (10) at (0, -2.25) {$\overset{\textcolor{red}{z}}{\ZZ} f$};
		\node [style=empty] (11) at (0, 2.25) {$\overset{\textcolor{red}{z}}{\ZZ} f$};
	\end{pgfonlayer}
	\begin{pgfonlayer}{edgelayer}
		\draw [style=to] (4) to (7);
		\draw [style=to] (5) to (6);
		\draw [style=to, in=180, out=0] (2) to (3);
		\draw [style=dashed-arrow] (0) to (1);
	\end{pgfonlayer}
\end{tikzpicture}\]
	
	It follows directly that $\overset{\textcolor{red}{z}}{\ZZ}$ maps the identity to the identity, 
	and that $\overset{\textcolor{red}{z}}{\iota}$ is natural.
	$\overset{\textcolor{red}{z}}{\ZZ}$ also preserves composition, which follows by the commutative diagram below:
	\[\begin{tikzpicture}
	\begin{pgfonlayer}{nodelayer}
		\node [style=empty] (0) at (-1.5, 3) {$\overset{\textcolor{red}{z}}{\ZZ} W$};
		\node [style=empty] (1) at (-1.5, 0) {$\overset{\textcolor{red}{z}}{\ZZ} X$};
		\node [style=empty] (2) at (1.5, 3) {$\overset{\textcolor{red}{z}}{\TT} W$};
		\node [style=empty] (3) at (1.5, 0) {$\overset{\textcolor{red}{z}}{\TT} X$};
		\node [style=empty] (4) at (-1, 3) {};
		\node [style=empty] (5) at (-1, 0) {};
		\node [style=empty] (6) at (1, 0) {};
		\node [style=empty] (7) at (1, 3) {};
		\node [style=empty] (8) at (0, 3.75) {$\iota_W$};
		\node [style=empty] (9) at (0, 0.75) {$\iota_X$};
		\node [style=empty] (10) at (2.25, 1.5) {$\overset{\textcolor{red}{z}}{\TT} g$};
		\node [style=empty] (11) at (-2.25, 1.5) {$\overset{\textcolor{red}{z}}{\ZZ} g$};
		\node [style=empty] (12) at (-1.5, -3) {$\overset{\textcolor{red}{z}}{\ZZ} Y$};
		\node [style=empty] (13) at (-2.25, -1.5) {$\overset{\textcolor{red}{z}}{\ZZ} f$};
		\node [style=empty] (14) at (-5.5, 0) {$\overset{\textcolor{red}{z}}{\ZZ}(f\circ g)$};
		\node [style=empty] (15) at (-1, -3) {};
		\node [style=empty] (16) at (1, -3) {};
		\node [style=empty] (17) at (1.5, -3) {$\overset{\textcolor{red}{z}}{\TT} Y$};
		\node [style=empty] (18) at (0, -2.25) {$\iota_Y$};
		\node [style=empty] (19) at (2.25, -1.5) {$\overset{\textcolor{red}{z}}{\TT} f$};
		\node [style=empty] (20) at (5.5, 0) {$\overset{\textcolor{red}{z}}{\TT}(f\circ g)$};
	\end{pgfonlayer}
	\begin{pgfonlayer}{edgelayer}
		\draw [style=to] (4) to (7);
		\draw [style=to] (5) to (6);
		\draw [style=to] (2) to (3);
		\draw [style=dashed-arrow, bend right=75] (0) to (12);
		\draw [style=to] (3) to (17);
		\draw [style=to] (15) to (16);
		\draw [style=to] (0) to (1);
		\draw [style=to] (1) to (12);
		\draw [style=to, bend left=75] (2) to (17);
	\end{pgfonlayer}
\end{tikzpicture}\]
	This proves that $\overset{\textcolor{red}{z}}{\ZZ}$ is a functor. 

	Second part.

	Then we describe its graded monad structure and prove all its morphisms exist and are unique.

	This part we give different proof strategy from that of \cite{central-submonads}, 
	which bypasses the Kleisli constructure and independent of premonoidal center.

	Definition of the monadic unit $\eta^Z_X$ 
	(the inequality arrow with respect to Definition \Ref{def:morph-monad-graded} becomes equal because now it's inclusion):
	\[\begin{tikzpicture}
	\begin{pgfonlayer}{nodelayer}
		\node [style=empty] (0) at (-2, 1) {$X$};
		\node [style=empty] (1) at (-1.5, -2) {$\overset{i}{\TT} X$};
		\node [style=empty] (2) at (1.75, 1) {$\overset{i}{\ZZ} X$};
		\node [style=empty] (3) at (-1.5, 1) {};
		\node [style=empty] (4) at (1.25, 1) {};
		\node [style=empty] (5) at (-1, -2) {};
		\node [style=empty] (6) at (1.75, 0.5) {};
		\node [style=empty] (7) at (-2.425, -0.65) {$\eta_X$};
		\node [style=empty] (8) at (2.5, -0.4) {$\iota_X$};
		\node [style=empty] (9) at (-0.025, 1.525) {$\eta^\ZZ_X$};
		\node [style=empty] (12) at (2, -2) {$\overset{i}{\TT} X$};
		\node [style=empty] (13) at (1.75, -1.5) {};
		\node [style=empty] (14) at (1.5, -2) {};
		\node [style=empty] (15) at (-1.5, 0.5) {};
		\node [style=empty] (16) at (-1.5, -1.5) {};
	\end{pgfonlayer}
	\begin{pgfonlayer}{edgelayer}
		\draw [style=dashed-arrow] (3) to (4);
		\draw [style=to] (6) to (13);
		\draw [style=to] (15) to (16);
		\draw [style=equal-arrow] (5) to (14);
	\end{pgfonlayer}
\end{tikzpicture}\]
	By Definition \ref{def:central-cone-graded}, 
	the universal property of terminal graded central cone indicates that, 
	for any graded central cone at $(\textcolor{red}{z}, X)$,
	there exists a unique morphism of graded central cones to $\overset{\textcolor{red}{z}}{\ZZ}$, 
	hence we need to prove that all other arrows in this definition form graded central cones.

	What's left is that $\eta_X$ forms a central cone, 
	it is proved by the following diagram:
	\[\stikz[0.7]{../figures/central-unit-proof-brut-graded.tikz}\]
	(1) $\eta$ property of graded co-strength; 
	(2) $\eta$ and $\tau$ are natural; (3),(4) definition of graded monad; 
	(5) $\TT$ is a functor with $\eta$ property of graded co-strength; 
	(6) $\tau$ is natural and rest are equalities.

	Next, definition of multiplication $\mu^\ZZ_X$ 
	(the inequality arrow with respect to Definition \Ref{def:morph-monad-graded} becomes equal because now it's inclusion):
	\[\begin{tikzpicture}
	\begin{pgfonlayer}{nodelayer}
		\node [style=empty] (19) at (-7.5, -6.75) {$\overset{\textcolor{red}{a}}{\ZZ}  \overset{\textcolor{teal}{b}}{\ZZ} X$};
		\node [style=empty] (20) at (-7.5, -11.75) {$\overset{\textcolor{red}{a}}{\TT}  \overset{\textcolor{teal}{b}}{\ZZ} X$};
		\node [style=empty] (21) at (-7.5, -7.25) {};
		\node [style=empty] (22) at (-7.5, -11.25) {};
		\node [style=empty] (23) at (-2.5, -11.75) {$\overset{\textcolor{red}{a}}{\TT}  \overset{\textcolor{teal}{b}}{\TT} X$};
		\node [style=empty] (24) at (7.5, -11.75) {$\overset{\textcolor{red}{a}\ast \textcolor{teal}{b}}{\TT}X$};
		\node [style=empty] (25) at (7.5, -6.75) {$\overset{\textcolor{red}{a}\ast \textcolor{teal}{b}}{\ZZ} X$};
		\node [style=empty] (26) at (7.5, -7.25) {};
		\node [style=empty] (27) at (0, -6) {$\mu^{\ZZ}_X$};
		\node [style=empty] (28) at (-5, -12.75) {$\overset{\textcolor{red}{a}}{\TT} \iota_X$};
		\node [style=empty] (29) at (0, -12.75) {$\mu_X^{\TT}$};
		\node [style=empty] (30) at (-8.5, -9.25) {$\iota_{\overset{\textcolor{teal}{b}}{\ZZ}X}$};
		\node [style=empty] (31) at (8.25, -9.25) {$\iota_X$};
		\node [style=empty] (32) at (7.5, -11.25) {};
		\node [style=empty] (33) at (2.5, -11.75) {$\overset{\textcolor{red}{a}\ast \textcolor{teal}{b}}{\TT}X$};
	\end{pgfonlayer}
	\begin{pgfonlayer}{edgelayer}
		\draw [style=to] (20) to (23);
		\draw [style=to] (21) to (22);
		\draw [style=dashed-arrow] (19) to (25);
		\draw [style=to] (26) to (32);
		\draw [style=to] (23) to (33);
		\draw [style=equal-arrow] (33) to (24);
	\end{pgfonlayer}
\end{tikzpicture}\]
	By Definition \ref{def:central-cone-graded}, 
	the universal property of terminal graded central cone indicates that, 
	for any graded central cone at $(\textcolor{red}{z}, X)$,
	there exists a unique morphism of graded central cones to $\overset{\textcolor{red}{z}}{\ZZ}$, 
	hence we need to prove that all other arrows in this definition form graded central cones.

	What's left is that $\mu_X^{\TT} \circ \overset{ \textcolor{red}{a}}{\TT} \iota_X \circ \iota_{\overset{\textcolor{teal}{b}}{\ZZ}X} $ forms a central cone, 
	it is proved by the following diagram:

	\[\stikz[0.4]{../figures/central-mult-proof-brut-graded-1.tikz}\]
	(1) $\iota_{\ZZ X}$ is central; (2) $\tau'$ and $\iota_X$ are natural; 
	(3) proposition \ref{prop:costrength-graded}, $\mu$ property of graded co-strength 
	(take normal monad as special case of graded monad); 
	(4) $\mu$ and $\tau$ are natural; (5) $\TT$ is a functor and $\tau$, $\iota_X$ are natural; 
	(6) $\TT$ is a functor and $\iota_X$ is central; (7) $\mu$ and $\iota_X$ are natural; 
	(8) $\tau$ and $\iota_X$ are natural; (9) $\TT$ is a functor and $\tau'$, $\iota_X$ are natural; 
	(10) $\mu$ and $\iota_X$ are natural; (11)  $\mu$ and $\tau'$ are natural;
	(12),(13) and (17) definition of monad; (14) $\mu$ and $\tau$ are natural;
	(15) $\TT$ is a functor and proposition \ref{prop:costrength-graded}, $\mu$ property of graded co-strength; 
	(16) $\mu$ and $\tau'$ are natural and rest are equalities.

	Last, definition of the strength $\tau^\ZZ_{W,X}$:
	\[\begin{tikzpicture}
	\begin{pgfonlayer}{nodelayer}
		\node [style=empty] (0) at (-3.5, 2) {$W\otimes\overset{\textcolor{red}{z}}{\ZZ} X$};
		\node [style=empty] (1) at (2.5, 2) {$\overset{\textcolor{red}{z}}{\ZZ}(W\otimes X)$};
		\node [style=empty] (2) at (-3.5, -1) {$W\otimes\overset{  \textcolor{red}{z}}{\TT} X$};
		\node [style=empty] (3) at (2.5, -1) {$\overset{  \textcolor{red}{z}}{\TT}(W\otimes X)$};
		\node [style=empty] (4) at (-3.5, 1.5) {};
		\node [style=empty] (5) at (-3.5, -0.5) {};
		\node [style=empty] (6) at (2.5, 1.5) {};
		\node [style=empty] (7) at (2.5, -0.5) {};
		\node [style=empty] (8) at (-4.75, 0.5) {$W\otimes\iota_X$};
		\node [style=empty] (9) at (3.5, 0.5) {$\iota_{W\otimes X}$};
		\node [style=empty] (10) at (-0.625, 2.625) {$\tau^\ZZ_{W,X}$};
		\node [style=empty] (11) at (-0.625, -1.5) {$\tau_{W,X}$};
	\end{pgfonlayer}
	\begin{pgfonlayer}{edgelayer}
		\draw [style=to] (4) to (5);
		\draw [style=to] (2) to (3);
		\draw [style=to] (6) to (7);
		\draw [style=dashed-arrow] (0) to (1);
	\end{pgfonlayer}
\end{tikzpicture}\]
	By Definition \ref{def:central-cone-graded}, 
	the universal property of terminal graded central cone indicates that, 
	for any graded central cone at $(\textcolor{red}{z}, X)$,
	there exists a unique morphism of graded central cones to $\overset{\textcolor{red}{z}}{\ZZ}$, 
	hence we need to prove that all other arrows in this definition form graded central cones.

	What's left is that $\tau_{W,X} \circ (W \otimes \iota_X) $ forms a central cone, 
	it is proved by the following diagram:
	\[\stikz[0.4]{../figures/central-strength-proof-brut-graded.tikz}\]
	(1), (3) $\alpha$, $\iota_X$ are natural; (2) lemma \ref{lem:tau}; (4) $\iota_X$ is central; 
	(5), (7) $\tau$ is natural; (6) $T$ is a functor, $\alpha \circ \alpha^{-1} = id$, definition on graded strength; 
	(8), (10), (11) definition of strength; (9)  $\tau$, $\tau'$ are natural; (12), (14) $\mu$ and $\alpha^{-1}$ are natural; 
	(13) $T$ is a functor, lemma \ref{lem:tau}; and rest are equalities.

	The rest of the proof is following the same strategy as that of centre Theorem in 
	\cite{central-submonads}.

	The last three definitions are exactly those of a morphism of strong 
	graded monads (see Definition \ref{def:morph-monad-graded}).  \\
	Using the fact that $\iota$ is monic (see Lemma~\ref{lem:monic}),
	the following commutative diagram shows that $\eta^\ZZ$ is natural:
	\[\scalebox{0.7}{\begin{tikzpicture}
	\begin{pgfonlayer}{nodelayer}
		\node [style=empty] (0) at (-6, 5) {$X$};
		\node [style=empty] (1) at (6, 5) {$\overset{i}{\ZZ} X$};
		\node [style=empty] (2) at (-6, -3) {$Y$};
		\node [style=empty] (3) at (6, -3) {$\overset{i}{\TT}Y$};
		\node [style=empty] (4) at (-6, 4.5) {};
		\node [style=empty] (5) at (-6, -2.5) {};
		\node [style=empty] (6) at (6, 4.5) {};
		\node [style=empty] (7) at (6, 1) {$\overset{i}{\ZZ} Y$};
		\node [style=empty] (8) at (6, 1.5) {};
		\node [style=empty] (9) at (6, 0.5) {};
		\node [style=empty] (10) at (6, -2.5) {};
		\node [style=empty] (11) at (0, -3) {$\overset{i}{\ZZ}Y$};
		\node [style=empty] (13) at (-2, 3.25) {$\overset{i}{\TT} X$};
		\node [style=empty] (14) at (1.75, 3.25) {$\overset{ i}{\TT}  X$};
		\node [style=empty] (15) at (5.5, 4.5) {};
		\node [style=empty] (16) at (-5.5, 4.5) {};
		\node [style=empty] (17) at (1.75, 2.75) {};
		\node [style=empty] (18) at (5.5, -2.5) {};
		\node [style=empty] (19) at (0, 5.75) {$\eta^{\overset{i}{\ZZ}}_X$};
		\node [style=empty] (20) at (-7, 1.25) {$f$};
		\node [style=empty] (21) at (7, 3.25) {$\overset{i}{\ZZ}f$};
		\node [style=empty] (22) at (2.25, 0.5) {$\overset{ i}{\TT} f$};
		\node [style=empty] (23) at (-4.25, -1.5) {$\eta_Y$};
		\node [style=empty] (24) at (-3, -3.75) {$\eta^{\overset{i}{\ZZ}}_Y$};
		\node [style=empty] (25) at (3, -3.75) {$\iota_Y$};
		\node [style=empty] (26) at (6.75, -1) {$\iota_Y$};
		\node [style=empty] (27) at (-3.75, 3.25) {$\eta_X$};
		\node [style=empty] (28) at (3.75, 3.5) {$\iota_X$};
		\node [style=none] (29) at (0, 4.25) {\textcolor{blue}{(1)}};
		\node [style=none] (30) at (4.25, 2) {\textcolor{blue}{(2)}};
		\node [style=none] (31) at (-4, 1) {\textcolor{blue}{(3)}};
		\node [style=none] (32) at (-1.5, -2.25) {\textcolor{blue}{(4)}};
		\node [style=empty] (33) at (-2, -1) {$\overset{i}{\TT} Y$};
		\node [style=empty] (34) at (-2, 2.75) {};
		\node [style=empty] (35) at (-2, -0.5) {};
		\node [style=empty] (36) at (-1.5, 1) {$\overset{i}{\TT} f$};
		\node [style=none] (37) at (0.5, 1) {\textcolor{blue}{(5)}};
	\end{pgfonlayer}
	\begin{pgfonlayer}{edgelayer}
		\draw [style=to] (0) to (1);
		\draw [style=to] (6) to (8);
		\draw [style=to] (4) to (5);
		\draw [style=to] (2) to (11);
		\draw [style=mono] (11) to (3);
		\draw [style=mono] (9) to (10);
		\draw [style=to] (16) to (13);
		\draw [style=to] (15) to (14);
		\draw [style=to] (17) to (18);
		\draw [style=to] (5) to (33);
		\draw [style=to] (34) to (35);
		\draw [style=equal-arrow] (13) to (14);
		\draw [style=equal-arrow] (35) to (18);
	\end{pgfonlayer}
\end{tikzpicture}}\]
	  (1), (4) definition of $\eta^\ZZ$; (2) $\iota$ is natural; (3) $\eta$ is natural;
	  and (5) equality.
	  Thus, we have proven that for any $f:X\to
	  Y$, $\iota_Y\circ\ZZ f\circ\eta^\ZZ_X = \iota_Y\circ\eta^\ZZ_Y\circ f$.
	  Besides, $\iota$ is monic, thus $\ZZ f\circ\eta^\ZZ_X =\eta^\ZZ_Y\circ f$
	  which proves that $\eta^\ZZ$ is natural.  We will prove all the remaining
	  diagrams with the same reasoning. \\
	The following commutative diagram shows that $\mu^\ZZ$ is natural. 
	\[\scalebox{0.7}{\begin{tikzpicture}
	\begin{pgfonlayer}{nodelayer}
		\node [style=empty] (0) at (-11, 5) {$\overset{\textcolor{red}{a}}{\ZZ}  \overset{\textcolor{teal}{b}}{\ZZ} X$};
		\node [style=empty] (1) at (11.5, 5) {$\overset{\textcolor{red}{a}\ast \textcolor{teal}{b}}{\ZZ} X$};
		\node [style=empty] (2) at (-11, -7) {$\overset{\textcolor{red}{a}}{\ZZ}  \overset{\textcolor{teal}{b}}{\ZZ} Y$};
		\node [style=empty] (3) at (11.5, -7) {$\overset{ \textcolor{red}{a}\ast \textcolor{teal}{b}}{\TT} Y$};
		\node [style=empty] (4) at (-11, 4.5) {};
		\node [style=empty] (5) at (-11, -6.5) {};
		\node [style=empty] (6) at (0, -7) {$\overset{\textcolor{red}{a}\ast \textcolor{teal}{b}}{\ZZ} Y$};
		\node [style=empty] (7) at (11.5, -1) {$\overset{\textcolor{red}{a}\ast \textcolor{teal}{b}}{\ZZ} Y$};
		\node [style=empty] (8) at (11.5, 4.5) {};
		\node [style=empty] (9) at (11.5, -0.5) {};
		\node [style=empty] (10) at (11.5, -1.5) {};
		\node [style=empty] (11) at (11.5, -6.5) {};
		\node [style=empty] (12) at (-6.5, 2) {$\overset{ \textcolor{red}{a}}{\TT}  \overset{  \textcolor{teal}{b}}{\TT}  X$};
		\node [style=empty] (13) at (7, 2) {$\overset{ \textcolor{red}{a}\ast \textcolor{teal}{b}}{\TT} X$};
		\node [style=empty] (14) at (-10.5, 4.5) {};
		\node [style=empty] (15) at (-7, 2.25) {};
		\node [style=empty] (16) at (7.5, 2.25) {};
		\node [style=empty] (17) at (11, 4.5) {};
		\node [style=empty] (19) at (-10.5, -6.5) {};
		\node [style=empty] (20) at (-6.5, -3.25) {$\overset{ \textcolor{red}{a}}{\TT}  \overset{  \textcolor{teal}{b}}{\TT}  Y$};
		\node [style=empty] (23) at (-6.5, 1.5) {};
		\node [style=empty] (24) at (-6.5, -2.75) {};
		\node [style=empty] (25) at (7.25, 1.5) {};
		\node [style=empty] (26) at (11, -6.25) {};
		\node [style=empty] (27) at (-12.25, -1) {$\overset{\textcolor{red}{a}}{\ZZ}  \overset{\textcolor{teal}{b}}{\ZZ} f$};
		\node [style=empty] (28) at (12.5, 2) {$\overset{\textcolor{red}{a}\ast \textcolor{teal}{b}}{\ZZ}f$};
		\node [style=empty] (29) at (12.5, -4) {$\iota_Y$};
		\node [style=empty] (30) at (5.75, -7.75) {$\iota_Y$};
		\node [style=empty] (31) at (0, 5.75) {$\mu^{\ZZ,\textcolor{red}{a},\textcolor{teal}{b}}_X$};
		\node [style=empty] (32) at (-9, 2.75) {$\iota\iota_X$};
		\node [style=empty] (33) at (-9, -4.25) {$\iota\iota_Y$};
		\node [style=empty] (34) at (7.5, -1.75) {$\overset{ \textcolor{red}{a}\ast \textcolor{teal}{b}}{\TT} f$};
		\node [style=empty] (35) at (-3, 2.75) {$\mu^{\TT,   \textcolor{red}{a},   \textcolor{teal}{b}}_X$};
		\node [style=empty] (36) at (-7.75, -0.5) {$\overset{ \textcolor{red}{a}}{\TT}  \overset{  \textcolor{teal}{b}}{\TT}  f$};
		\node [style=empty] (37) at (-3, -2.5) {$\mu^{\TT,   \textcolor{red}{a},   \textcolor{teal}{b}}_Y$};
		\node [style=empty] (38) at (-5.5, -7.75) {$\mu^{\ZZ,\textcolor{red}{a},\textcolor{teal}{b}}_Y$};
		\node [style=none] (39) at (0, 3.5) {\textcolor{blue}{(1)}};
		\node [style=none] (40) at (-8.75, -2) {\textcolor{blue}{(2)}};
		\node [style=none] (41) at (-3, -0.5) {\textcolor{blue}{(3)}};
		\node [style=none] (42) at (9.75, 0.75) {\textcolor{blue}{(4)}};
		\node [style=none] (43) at (-2.75, -5) {\textcolor{blue}{(5)}};
		\node [style=empty] (46) at (0.75, -3.25) {$\overset{ \textcolor{red}{a}\ast   \textcolor{teal}{b}}{\TT} Y$};
		\node [style=empty] (47) at (0.75, 2) {$\overset{ \textcolor{red}{a}\ast   \textcolor{teal}{b}}{\TT} X$};
		\node [style=empty] (48) at (0.75, -2.75) {};
		\node [style=empty] (51) at (1.75, -0.5) {$\overset{ \textcolor{red}{a}\ast   \textcolor{teal}{b}}{\TT} f$};
		\node [style=empty] (52) at (9.75, 3) {$\iota_X$};
		\node [style=empty] (53) at (0.75, 1.5) {};
		\node [style=none] (54) at (4.75, -0.75) {\textcolor{blue}{(6)}};
	\end{pgfonlayer}
	\begin{pgfonlayer}{edgelayer}
		\draw [style=to] (4) to (5);
		\draw [style=mono] (10) to (11);
		\draw [style=to] (8) to (9);
		\draw [style=to] (2) to (6);
		\draw [style=mono] (6) to (3);
		\draw [style=to] (0) to (1);
		\draw [style=to] (14) to (15);
		\draw [style=to] (17) to (16);
		\draw [style=to] (19) to (20);
		\draw [style=to] (23) to (24);
		\draw [style=to] (25) to (26);
		\draw [style=to] (20) to (46);
		\draw [style=to] (12) to (47);
		\draw [style=to] (53) to (48);
		\draw [style=equal-arrow] (47) to (13);
		\draw [style=equal-arrow] (46) to (26);
	\end{pgfonlayer}
\end{tikzpicture}}\]
	(1) (5) definition of $\mu^\ZZ$; (2) (4) $\iota$ is natural; (3) $\mu$ is natural and (6) equalities.

	The commutative diagrams showing that $\tau^\ZZ$ is natural are just as
	those on normal monad by replacing $\ZZ$ as $\overset{\textcolor{red}{z}}{\ZZ}$, 
	$\TT$ as $\overset{ \textcolor{red}{z}}{\TT}$ and $\iota$ as family of morphisms
	$\overset{\textcolor{red}{z}}{\iota} :
	\overset{\textcolor{red}{z}}{\ZZ} \to \overset{ \textcolor{red}{z}}{\TT}$ indexed by elements in $\GG$.

	The following commutative diagrams show that $\tau^\ZZ$ is natural.
	\[\scalebox{0.7}{\begin{tikzpicture}
	\begin{pgfonlayer}{nodelayer}
		\node [style=empty] (0) at (-13, 6) {$A\otimes\ZZ C$};
		\node [style=empty] (1) at (5, 6) {$\ZZ(A\otimes C)$};
		\node [style=empty] (2) at (-13, -2) {$B\otimes\ZZ C$};
		\node [style=empty] (3) at (5, -2) {$\TT(B\otimes C)$};
		\node [style=empty] (4) at (-6, -2) {$\ZZ(B\otimes C)$};
		\node [style=empty] (5) at (5, 2) {$\ZZ(B\otimes C)$};
		\node [style=empty] (6) at (5, 5.5) {};
		\node [style=empty] (7) at (5, 2.5) {};
		\node [style=empty] (8) at (5, 1.5) {};
		\node [style=empty] (9) at (5, -1.5) {};
		\node [style=empty] (10) at (-13, 5.5) {};
		\node [style=empty] (11) at (-13, -1.5) {};
		\node [style=empty] (12) at (-7, 0.5) {$B\otimes\TT C$};
		\node [style=empty] (13) at (-7, 4) {$A\otimes\TT C$};
		\node [style=empty] (14) at (-1, 4) {$\TT(A\otimes C)$};
		\node [style=empty] (15) at (-7, 3.5) {};
		\node [style=empty] (16) at (-7, 1) {};
		\node [style=empty] (17) at (-5.5, 0.25) {};
		\node [style=empty] (18) at (3.5, -1.5) {};
		\node [style=empty] (19) at (4.25, -1.25) {};
		\node [style=empty] (20) at (-0.75, 3.5) {};
		\node [style=empty] (21) at (-12.5, 5.5) {};
		\node [style=empty] (22) at (-8.5, 4.25) {};
		\node [style=empty] (23) at (4.5, 5.5) {};
		\node [style=empty] (24) at (0.5, 4.25) {};
		\node [style=empty] (25) at (-12.5, -1.75) {};
		\node [style=empty] (26) at (-8.25, 0) {};
		\node [style=empty] (27) at (-4, 6.75) {$\tau^\ZZ_{A,C}$};
		\node [style=empty] (28) at (-4, 4.5) {$\tau_{A,C}$};
		\node [style=empty] (29) at (-9.5, -2.75) {$\tau^\ZZ_{B,C}$};
		\node [style=empty] (30) at (-1, -2.75) {$\iota_{B\otimes C}$};
		\node [style=empty] (31) at (6, 0.25) {$\iota_{B\otimes C}$};
		\node [style=empty] (32) at (-3, -0.75) {$\tau_{B,C}$};
		\node [style=empty] (33) at (-11.5, 2) {$f\otimes\ZZ C$};
		\node [style=empty] (34) at (-9.25, 5) {$A\otimes\iota$};
		\node [style=empty] (35) at (-9.25, -1) {$B\otimes\iota$};
		\node [style=empty] (36) at (6.5, 4) {$\ZZ(f\otimes C)$};
		\node [style=empty] (37) at (0, 0.75) {$\TT(f\otimes C)$};
		\node [style=empty] (38) at (-5.5, 2.25) {$f\otimes\TT C$};
		\node [style=empty] (39) at (2.75, 4.5) {$\iota$};
		\node [style=none] (40) at (-4, 5.25) {\textcolor{blue}{(1)}};
		\node [style=none] (41) at (-9, 2) {\textcolor{blue}{(2)}};
		\node [style=none] (42) at (-2.5, 2) {\textcolor{blue}{(3)}};
		\node [style=none] (43) at (2.5, 2.25) {\textcolor{blue}{(4)}};
		\node [style=none] (44) at (-6, -1) {\textcolor{blue}{(5)}};
	\end{pgfonlayer}
	\begin{pgfonlayer}{edgelayer}
		\draw [style=to] (0) to (1);
		\draw [style=to] (10) to (11);
		\draw [style=to] (21) to (22);
		\draw [style=to] (23) to (24);
		\draw [style=to] (13) to (14);
		\draw [style=to] (15) to (16);
		\draw [style=to] (25) to (26);
		\draw [style=to] (2) to (4);
		\draw [style=to] (17) to (18);
		\draw [style=to] (20) to (19);
		\draw [style=to] (6) to (7);
		\draw [style=mono] (4) to (3);
		\draw [style=mono] (8) to (9);
	\end{pgfonlayer}
\end{tikzpicture}}\]
	(1) definition of $\tau^\ZZ$, (2) $\iota$ is natural, (3) $\tau$ is natural, (4) $\iota$ is natural
	and (5) definition of $\tau^\ZZ$. \\
	\[\scalebox{0.7}{\begin{tikzpicture}
	\begin{pgfonlayer}{nodelayer}
		\node [style=empty] (0) at (-9, 4.25) {$A\otimes\ZZ B$};
		\node [style=empty] (1) at (9, 4.25) {$\ZZ(A\otimes B)$};
		\node [style=empty] (2) at (-9, -3.75) {$A\otimes\ZZ C$};
		\node [style=empty] (3) at (9, -3.75) {$\TT(A\otimes C)$};
		\node [style=empty] (4) at (-2, -3.75) {$\ZZ(A\otimes C)$};
		\node [style=empty] (5) at (9, 0.25) {$\ZZ(A\otimes C)$};
		\node [style=empty] (6) at (9, 3.75) {};
		\node [style=empty] (7) at (9, 0.75) {};
		\node [style=empty] (8) at (9, -0.25) {};
		\node [style=empty] (9) at (9, -3.25) {};
		\node [style=empty] (10) at (-9, 3.75) {};
		\node [style=empty] (11) at (-9, -3.25) {};
		\node [style=empty] (12) at (-3, -1.25) {$A\otimes\TT C$};
		\node [style=empty] (13) at (-3, 2.25) {$A\otimes\TT B$};
		\node [style=empty] (14) at (3, 2.25) {$\TT(A\otimes B)$};
		\node [style=empty] (15) at (-3, 1.75) {};
		\node [style=empty] (16) at (-3, -0.75) {};
		\node [style=empty] (17) at (-1.5, -1.5) {};
		\node [style=empty] (18) at (7.5, -3.25) {};
		\node [style=empty] (19) at (8.25, -3) {};
		\node [style=empty] (20) at (3.25, 1.75) {};
		\node [style=empty] (21) at (-8.5, 3.75) {};
		\node [style=empty] (22) at (-4.5, 2.5) {};
		\node [style=empty] (23) at (8.5, 3.75) {};
		\node [style=empty] (24) at (4.5, 2.5) {};
		\node [style=empty] (25) at (-8.5, -3.5) {};
		\node [style=empty] (26) at (-4.25, -1.75) {};
		\node [style=empty] (27) at (0, 5) {$\tau^\ZZ_{A,B}$};
		\node [style=empty] (28) at (0, 2.75) {$\tau_{A,B}$};
		\node [style=empty] (29) at (-5.5, -4.5) {$\tau^\ZZ_{A,C}$};
		\node [style=empty] (30) at (3, -4.5) {$\iota_{A\otimes C}$};
		\node [style=empty] (31) at (10, -1.5) {$\iota_{A\otimes C}$};
		\node [style=empty] (32) at (1, -2.5) {$\tau_{A,C}$};
		\node [style=empty] (33) at (-7.5, 0.25) {$A\otimes\ZZ f$};
		\node [style=empty] (34) at (-5.25, 3.25) {$A\otimes\iota$};
		\node [style=empty] (35) at (-5.25, -2.75) {$A\otimes\iota$};
		\node [style=empty] (36) at (10.5, 2.25) {$\ZZ(A\otimes f)$};
		\node [style=empty] (37) at (4, -1) {$\TT(A\otimes f)$};
		\node [style=empty] (38) at (-1.5, 0.5) {$A\otimes\TT f$};
		\node [style=empty] (39) at (6.75, 2.75) {$\iota$};
		\node [style=none] (40) at (0, 3.5) {\textcolor{blue}{(1)}};
		\node [style=none] (41) at (-5, 0.25) {\textcolor{blue}{(2)}};
		\node [style=none] (42) at (1.5, 0.25) {\textcolor{blue}{(3)}};
		\node [style=none] (43) at (6.25, 1) {\textcolor{blue}{(4)}};
		\node [style=none] (44) at (-1.75, -2.75) {\textcolor{blue}{(5)}};
	\end{pgfonlayer}
	\begin{pgfonlayer}{edgelayer}
		\draw [style=to] (0) to (1);
		\draw [style=to] (10) to (11);
		\draw [style=to] (21) to (22);
		\draw [style=to] (23) to (24);
		\draw [style=to] (13) to (14);
		\draw [style=to] (15) to (16);
		\draw [style=to] (25) to (26);
		\draw [style=to] (2) to (4);
		\draw [style=to] (17) to (18);
		\draw [style=to] (20) to (19);
		\draw [style=to] (6) to (7);
		\draw [style=mono] (4) to (3);
		\draw [style=mono] (8) to (9);
	\end{pgfonlayer}
\end{tikzpicture}}\]
	(1) definition of $\tau^\ZZ$, (2) $\iota$ is natural, (3) $\tau$ is natural,
	(4) $\iota$ is natural and (5) definition of $\tau^\ZZ$. 
	\\ 
	The following commutative diagrams prove that $\ZZ$ is a graded monad 
	($\overset{\textcolor{red}{z}}{\ZZ} \cdot I $ is omitted because it is very similar to $I\cdot \overset{\textcolor{red}{z}}{\ZZ} $).
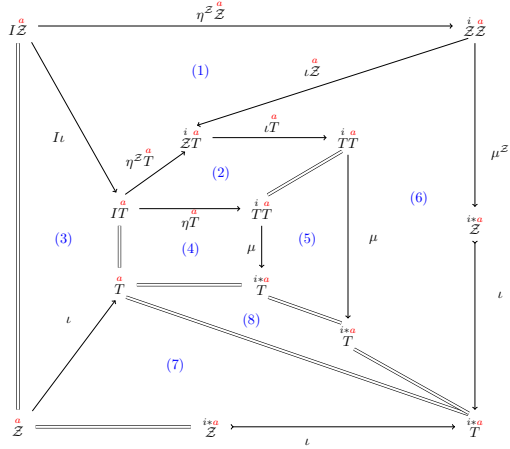
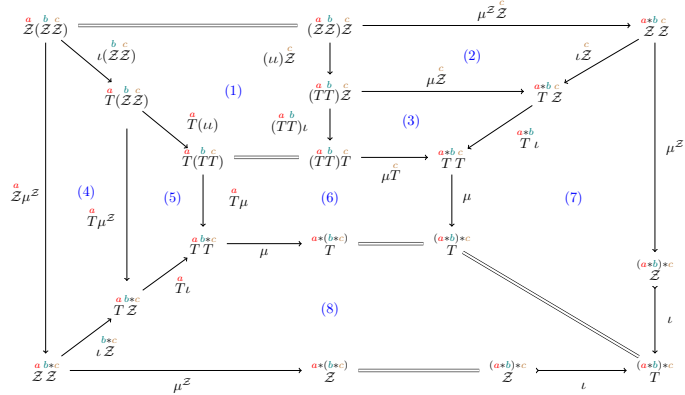
	\[\scalebox{0.6}{\input{z-monad-1-graded.tikz}}\]
	Top : (1)  $\eta^\ZZ$ and $\iota$ are natural; (2) definition of $\eta^{\ZZ}$,
	(3) (5) (7) (8) equality; (4) definition of $\eta$ and (6) definition of $\mu^{\ZZ}$. \\
	Down : (1) equality; (2) definition of $\mu^\ZZ$, $\ZZ$ is a functor; 
	(3) $\mu$ and $\iota$ are functors; (4) $\mu^\ZZ$ and $\iota$ are functors; 
	(5) (7) (8) definition of $\mu^\ZZ$, $\TT$ is a functor and 
	(6) definition of graded monad. \\

	Last Part:

	$\ZZ$ is proven strong with very similar diagrams. \\
	The following commutative diagram proves that $\ZZ$ is a commutative graded monad:
	  \begin{equation}\label{eq:proof-commutative-graded}
		  \scalebox{0.6}{\begin{tikzpicture}
	\begin{pgfonlayer}{nodelayer}
		\node [style=empty] (0) at (-10, 6) {$\overset{\textcolor{red}{a}}{\ZZ} X \otimes\overset{\textcolor{teal}{b}}{\ZZ} Y $};
		\node [style=empty] (1) at (24.5, 6) {$\overset{\textcolor{red}{a}}{\ZZ}    \overset{\textcolor{teal}{b}}{\ZZ} (X\otimes Y)$};
		\node [style=empty] (2) at (-10, -20) {$\overset{\textcolor{teal}{b}}{\ZZ}  \overset{\textcolor{red}{a}}{\ZZ}(X\otimes Y)$};
		\node [style=empty] (3) at (14.25, -20) {$\overset{ (\textcolor{teal}{b}\ast \textcolor{red}{a})}{\TT}(X\otimes Y)$};
		\node [style=empty] (4) at (-10, -10) {$\overset{\textcolor{teal}{b}}{\ZZ} (\overset{\textcolor{red}{a}}{\ZZ}  X\otimes Y)$};
		\node [style=empty] (5) at (24.5, -1.5) {$\overset{\textcolor{red}{a}\ast  \textcolor{teal}{b}}{\ZZ}(X\otimes Y)$};
		\node [style=empty] (6) at (5, -20) {$\overset{\textcolor{teal}{b}\ast \textcolor{red}{a}}{\ZZ}(X\otimes Y)$};
		\node [style=empty] (7) at (10.25, 6) {$\overset{\textcolor{red}{a}}{\ZZ} ( X \otimes\overset{\textcolor{teal}{b}}{\ZZ} Y )$};
		\node [style=empty] (8) at (-10, 5.5) {};
		\node [style=empty] (9) at (-10, -9.5) {};
		\node [style=empty] (10) at (-10, -10.5) {};
		\node [style=empty] (11) at (-10, -19.5) {};
		\node [style=empty] (12) at (24.5, 5.5) {};
		\node [style=empty] (13) at (24.5, -1) {};
		\node [style=empty] (14) at (24.5, -2) {};
		\node [style=empty] (15) at (24.5, -6.75) {};
		\node [style=empty] (16) at (1.5, 6.75) {$\tau'$};
		\node [style=empty] (17) at (17.75, 6.75) {$\overset{\textcolor{red}{a}}{\ZZ}\tau$};
		\node [style=empty] (18) at (26.25, 2) {$\mu^\ZZ_{X\otimes Y}$};
		\node [style=empty] (19) at (-1.5, -20.75) {$\mu^\ZZ_{X\otimes Y}$};
		\node [style=empty] (20) at (-11.25, -15) {$\overset{\textcolor{teal}{b}}{\ZZ}\tau'$};
		\node [style=empty] (21) at (-11.25, -2.25) {$\tau$};
		\node [style=empty] (22) at (-3.5, -17) {$\overset{\textcolor{teal}{b}}{\ZZ}  \overset{ \textcolor{red}{a}}{\TT}(X\otimes Y)$};
		\node [style=empty] (23) at (5.5, -17) {$\overset{ \textcolor{teal}{b}}{\TT}  \overset{ \textcolor{red}{a}}{\TT}(X\otimes Y)$};
		\node [style=empty] (24) at (-9.25, -19.5) {};
		\node [style=empty] (25) at (-4.25, -17.5) {};
		\node [style=empty] (26) at (8.25, -17) {};
		\node [style=empty] (28) at (19.5, 2) {$\overset{\textcolor{red}{a}}{\ZZ}\overset{ \textcolor{teal}{b}}{\TT} (X\otimes Y)$};
		\node [style=empty] (29) at (19.5, -5.25) {$\overset{ \textcolor{red}{a}}{\TT}  \overset{ \textcolor{teal}{b}}{\TT}(X\otimes Y)$};
		\node [style=empty] (30) at (19.5, 1.5) {};
		\node [style=empty] (31) at (19.5, -4.75) {};
		\node [style=empty] (32) at (24, 5.5) {};
		\node [style=empty] (33) at (19.75, 2.5) {};
		\node [style=empty] (35) at (19.5, -5.75) {};
		\node [style=empty] (36) at (-3.5, -10) {$ \overset{\textcolor{teal}{b}}{\ZZ}( \overset{ \textcolor{red}{a}}{\TT} X\otimes Y)$};
		\node [style=empty] (39) at (-3.5, -10.5) {};
		\node [style=empty] (40) at (-3.5, -16.5) {};
		\node [style=empty] (41) at (10.25, 2) {$\overset{\textcolor{red}{a}}{\ZZ} (X\otimes\overset{ \textcolor{teal}{b}}{\TT}  Y)$};
		\node [style=empty] (42) at (10.25, 5.5) {};
		\node [style=empty] (43) at (10.25, 2.5) {};
		\node [style=empty] (44) at (5.5, -10) {$ \overset{ \textcolor{teal}{b}}{\TT}( \overset{ \textcolor{red}{a}}{\TT} X\otimes Y)$};
		\node [style=empty] (47) at (5.5, -10.5) {};
		\node [style=empty] (48) at (5.5, -16.5) {};
		\node [style=empty] (49) at (10.25, -5.25) {$\overset{ \textcolor{red}{a}}{\TT}(X\otimes\overset{ \textcolor{teal}{b}}{\TT} Y)$};
		\node [style=empty] (50) at (10.25, 1.5) {};
		\node [style=empty] (51) at (10.25, -4.75) {};
		\node [style=empty] (54) at (4.25, -2.75) {$\overset{ \textcolor{red}{a}}{\TT}  X\otimes\overset{ \textcolor{teal}{b}}{\TT} Y$};
		\node [style=empty] (55) at (2.25, -4.25) {};
		\node [style=empty] (56) at (5.5, -9.5) {};
		\node [style=empty] (57) at (-3.5, -3.75) {$\overset{ \textcolor{red}{a}}{\TT} X\otimes\overset{\textcolor{teal}{b}}{\ZZ}  Y$};
		\node [style=empty] (58) at (-1, 2) {$\overset{\textcolor{red}{a}}{\ZZ}  X\otimes \overset{ \textcolor{teal}{b}}{\TT} Y$};
		\node [style=empty] (60) at (-3.75, -3.25) {};
		\node [style=empty] (62) at (-1.75, 2.25) {};
		\node [style=empty] (63) at (-3.5, -4.25) {};
		\node [style=empty] (64) at (-3.5, -9.5) {};
		\node [style=empty] (65) at (-0.75, 1.5) {};
		\node [style=empty] (66) at (4, -2.25) {};
		\node [style=empty] (67) at (-5.25, 4.25) {$\iota$};
		\node [style=empty] (68) at (-7, 0.25) {$\iota$};
		\node [style=empty] (69) at (-6.25, -18.75) {$\iota$};
		\node [style=empty] (70) at (22, 3.5) {$\iota$};
		\node [style=empty] (71) at (-0.75, -4.25) {$\iota$};
		\node [style=empty] (72) at (2.25, 0) {$\iota$};
		\node [style=empty] (73) at (9.25, 4) {$\iota$};
		\node [style=empty] (74) at (5, 1) {$\tau'$};
		\node [style=empty] (75) at (-5, -6.75) {$\tau$};
		\node [style=empty] (76) at (-6.75, -9) {$\iota$};
		\node [style=empty] (77) at (-5, -13.5) {$\overset{\textcolor{teal}{b}}{\ZZ} \tau'$};
		\node [style=empty] (78) at (1, -10.75) {$\iota$};
		\node [style=empty] (79) at (5, -6.75) {$\tau$};
		\node [style=empty] (80) at (6.5, -4.5) {$\tau'$};
		\node [style=empty] (81) at (7.25, -13.5) {$ \overset{ \textcolor{teal}{b}}{\TT}\tau'$};
		\node [style=empty] (82) at (14.75, -6.25) {$\overset{ \textcolor{red}{a}}{\TT}\tau$};
		\node [style=empty] (83) at (1, -16.5) {$\iota$};
		\node [style=empty] (84) at (20.5, -9.75) {$\mu_{X\otimes Y}$};
		\node [style=empty] (85) at (15.5, -17.5) {$\mu_{X\otimes Y}$};
		\node [style=empty] (86) at (20.25, -1.5) {$\iota$};
		\node [style=empty] (87) at (9.25, -1.5) {$\iota$};
		\node [style=empty] (88) at (15, 1) {$\overset{\textcolor{red}{a}}{\ZZ} \tau$};
		\node [style=empty] (89) at (26.25, -4.25) {$\iota_{X\otimes Y}$};
		\node [style=empty] (90) at (9.5, -20.75) {$\iota_{X\otimes Y}$};
		\node [style=none] (91) at (2.5, 4) {\textcolor{blue}{(1)}};
		\node [style=none] (92) at (16, 4) {\textcolor{blue}{(2)}};
		\node [style=none] (93) at (-6.75, -4.5) {\textcolor{blue}{(3)}};
		\node [style=none] (94) at (-3.25, 0.25) {\textcolor{blue}{(4)}};
		\node [style=none] (95) at (6.75, -0.75) {\textcolor{blue}{(5)}};
		\node [style=none] (96) at (15, -1.75) {\textcolor{blue}{(6)}};
		\node [style=none] (97) at (22.25, -2.5) {\textcolor{blue}{(7)}};
		\node [style=none] (98) at (0, -6.75) {\textcolor{blue}{(8)}};
		\node [style=none] (99) at (14.25, -12.25) {\textcolor{blue}{(9)}};
		\node [style=none] (100) at (-7, -14.25) {\textcolor{blue}{(10)}};
		\node [style=none] (101) at (1, -13.5) {\textcolor{blue}{(11)}};
		\node [style=none] (102) at (4.25, -18.5) {\textcolor{blue}{(12)}};
		\node [style=empty] (103) at (2, -3.75) {$\overset{ \textcolor{red}{a}}{\TT} X\otimes\overset{ \textcolor{teal}{b}}{\TT}  Y$};
		\node [style=empty] (106) at (-1.25, 1.5) {};
		\node [style=empty] (107) at (1.75, -3.25) {};
		\node [style=empty] (108) at (-0.75, -1.5) {$\iota$};
		\node [style=empty] (109) at (24.5, -13.25) {$\overset{ \textcolor{red}{a}\ast \textcolor{teal}{b}}{\TT}(X\otimes Y)$};
		\node [style=empty] (110) at (24.5, -13.75) {};
		\node [style=empty] (111) at (24.5, -19.5) {};
		\node [style=empty] (112) at (24.5, -20) {$\overset{ \textcolor{teal}{b}\ast \textcolor{red}{a}}{\TT}(X\otimes Y)$};
		\node [style=empty] (113) at (24.5, -7.75) {};
		\node [style=empty] (114) at (24.5, -7.25) {$\overset{ \textcolor{red}{a}\ast  \textcolor{teal}{b}}{\TT}(X\otimes Y)$};
		\node [style=empty] (115) at (24.5, -12.75) {};
	\end{pgfonlayer}
	\begin{pgfonlayer}{edgelayer}
		\draw [style=to] (0) to (7);
		\draw [style=to] (7) to (1);
		\draw [style=to] (12) to (13);
		\draw [style=to] (2) to (6);
		\draw [style=to] (10) to (11);
		\draw [style=to] (8) to (9);
		\draw [style=to] (22) to (23);
		\draw [style=to] (24) to (25);
		\draw [style=to] (30) to (31);
		\draw [style=to] (32) to (33);
		\draw [style=to] (39) to (40);
		\draw [style=to] (42) to (43);
		\draw [style=to] (41) to (28);
		\draw [style=to] (47) to (48);
		\draw [style=to] (50) to (51);
		\draw [style=to] (55) to (56);
		\draw [style=to] (54) to (49);
		\draw [style=to] (63) to (64);
		\draw [style=to] (65) to (66);
		\draw [style=to] (58) to (41);
		\draw [style=mono] (6) to (3);
		\draw [style=mono] (14) to (15);
		\draw [style=to] (106) to (107);
		\draw [style=equal-arrow] (110) to (111);
		\draw [style=to] (8) to (60);
		\draw [style=to] (8) to (62);
		\draw [style=to] (4) to (36);
		\draw [style=to] (36) to (44);
		\draw [style=to] (49) to (29);
		\draw [style=to] (57) to (103);
		\draw [style=to] (26) to (111);
		\draw [style=to] (35) to (115);
		\draw [style=equal-arrow] (3) to (112);
		\draw [style=equal-arrow] (113) to (115);
	\end{pgfonlayer}
\end{tikzpicture}}
	  \end{equation}
	  (1) $\tau'^\ZZ$ is natural;
	  (2) definition of $\tau^\ZZ$; (3) $\tau^\ZZ$ is natural; (4) $\CC$ is
	  monoidal; (5) definition of $\tau'^\ZZ$; (6) $\iota$ is natural; (7), (12)
	  definition of $\mu^\ZZ$; (8) definition of $\tau^\ZZ$; (9) $\iota$ is
	  central; (10) definition of $\tau'^\ZZ$; (11) $\iota$ is natural.
\end{proof}

\end{document}